\documentclass[a4paper, 11pt]{article}

\usepackage[utf8]{inputenc}
\usepackage[T1]{fontenc}
\usepackage[english]{babel} % English language setting

\usepackage[margin=1in]{geometry} % Standard 1-inch margins
\usepackage{hyperref} % For clickable links
\hypersetup{
    colorlinks=true,
    linkcolor=blue,
    citecolor=darkgray,
    urlcolor=blue
}
\usepackage{xcolor}

\usepackage{amsmath, amssymb, amsthm, amsfonts}
\usepackage{mathtools} % Additional math tools
\usepackage{bm}        % Bold math symbols

\usepackage{graphicx}
\usepackage{subcaption}
\usepackage{float}      % To enforce position with [H]

\usepackage{algorithm}
\usepackage{algpseudocode}
\usepackage{tikz}                 % Charge TikZ
\usepackage{amssymb}
\usepackage{algpseudocode}
\usetikzlibrary{arrows.meta, calc, decorations.markings}
\usepackage{booktabs}

\theoremstyle{plain}
\newtheorem{theorem}{Theorem}[section]
\newtheorem{lemma}[theorem]{Lemma}
\newtheorem{proposition}[theorem]{Proposition}
\newtheorem{corollary}[theorem]{Corollary}

\theoremstyle{definition}
\newtheorem{definition}{Definition}[section]
\newtheorem{assumption}{Assumption}

\theoremstyle{remark}
\newtheorem{remark}{Remark}[section]

\newcommand{\R}{\mathbb{R}}
\newcommand{\N}{\mathbb{N}}
\newcommand{\St}{St(n,p)} % Stiefel Manifold
\newcommand{\Gr}{Gr(n,p)} % Grassmannian
\newcommand{\trace}{\mathrm{tr}}
\newcommand{\Sym}{\mathrm{Sym}}

\DeclareMathOperator{\Image}{Im}

\DeclareMathOperator{\rank}{rank}
\DeclareMathOperator{\Span}{Span}

\DeclareMathOperator{\diag}{diag}

\newcommand{\so}{\mathfrak{so}}

\title{\textbf{Grassmannian Geometry and Global Convergence of Variable Projection for Neural Networks}}

\author{
    Mathias DUS \thanks{IRMA, Strasbourg, France (\texttt{mathias.dus@math.unistra.fr}).} 
}
\date{\today}

\begin{document}

\maketitle

% --- Abstract ---
\begin{abstract}
Training deep neural networks and Physics-Informed Neural Networks (PINNs) often leads to ill-conditioned and stiff optimization problems. A key structural feature of these models is that they are linear in the output-layer parameters and nonlinear in the hidden-layer parameters, yielding a separable nonlinear least-squares formulation. In this work, we study the classical variable projection (VarPro) method for such problems in the context of deep neural networks. We provide a geometric formulation on the Grassmannian and analyze the structure of critical points and convergence properties of the reduced problem. When the feature map is parametrized by a neural network, we show that these properties persist except in rank-deficient regimes, which we address via a regularized Grassmannian framework. Numerical experiments for regression and PINNs, including an efficient solver for the heat equation, illustrate the practical effectiveness of the approach.

    \vspace{0.5em}
    
    \vspace{0.5em}
\end{abstract}

% --- Content ---

\section{Introduction}

\subsection{Litterature review}

The approximation of high-dimensional functions by deep neural networks (DNNs) lies at the core of modern machine learning and scientific computing. In recent years, this paradigm has been extended to the numerical solution of partial differential equations (PDEs) through Physics-Informed Neural Networks (PINNs), which recast forward and inverse PDE problems as constrained empirical risk minimization problems \cite{Raissi2019}. Despite their flexibility and expressive power, PINNs and related deep-learning-based solvers often suffer from severe optimization pathologies, including stiffness, poor conditioning, and slow or unstable convergence, especially when multiple physical constraints are enforced simultaneously \cite{Wang2021}.

From an optimization standpoint, the training of many neural network architectures admits a natural \emph{separable} structure. In particular, for standard feedforward networks and PINNs, the model output is linear with respect to the parameters of the final layer and nonlinear with respect to the hidden-layer parameters. Denoting by $\alpha \in \mathbb{R}^p$ the output-layer weights and by $\theta$ the collection of nonlinear parameters, the network output can be written as
\begin{equation}\label{eq:separable_model}
    f(x;\alpha,\theta) = \Phi(x;\theta)\,\alpha,
\end{equation}
where $\Phi(\theta)$ denotes the nonlinear feature map induced by the hidden layers. Training then amounts to solving a separable nonlinear least-squares (SNLS) problem of the form
\begin{equation}\label{eq:snls}
    \min_{\alpha,\theta} 
    \mathcal{L}(\alpha,\theta)
    = \tfrac{1}{2}\|y-\Phi(\theta)\alpha\|_2^2
    + \lambda \mathcal{R}(\alpha,\theta),
\end{equation}
where $\mathcal{R}$ may include regularization terms or physics-based residuals, as in PINNs.

Standard optimization strategies in machine learning—such as stochastic gradient descent (SGD) and adaptive variants like Adam—treat all parameters on equal footing and ignore this structural separation. While these methods are broadly applicable, their convergence analyses rely on assumptions such as smoothness, bounded gradients, or favorable noise properties that are often violated in highly nonconvex and ill-conditioned problems arising in scientific machine learning \cite{Bottou2018}. In particular, the interaction between poorly conditioned feature maps $\Phi(\theta)$ and gradient-based updates can exacerbate stiffness and slow convergence, a phenomenon observed empirically and analyzed in the PINN literature \cite{Wang2021}.

A classical alternative for SNLS problems is the \emph{variable projection} (VarPro) method, introduced by Golub and Pereyra \cite{Golub1973}. The key idea is to explicitly eliminate the linear variables by solving
\[
    \alpha(\theta) = \arg\min_{\alpha} \mathcal{L}(\alpha,\theta),
\]
which yields $\alpha(\theta) = \Phi(\theta)^{\dagger}y$ in the unregularized case, where $\dagger$ denotes the Moore--Penrose pseudoinverse. Substituting this expression into the loss function leads to a \emph{reduced functional} depending solely on the nonlinear parameters,
\[
    r(\theta)
    = \tfrac{1}{2}\|(I - \Phi(\theta)\Phi(\theta)^{\dagger})y\|_2^2,
\]
whose gradient and Hessian incorporate the sensitivity of the projection operator with respect to $\theta$. Classical analyses have shown that VarPro can significantly improve convergence behavior and numerical stability compared to alternating minimization strategies, and that this improvement can be understood through the spectral properties of the reduced problem \cite{Golub1973,Ruhe1980}. Practical algorithmic realizations were developed in \cite{Kaufman1975}.

Despite its long-standing success in numerical linear algebra, the theoretical behavior of VarPro in the context of deep neural networks remains insufficiently understood. The primary challenge stems from the fact that the feature matrix $\Phi(\theta)$ depends nonlinearly and nonconvexly on $\theta$, is often highly over-parameterized, and may become rank-deficient during training. While early work explored projection-based training for neural networks \cite{Pereyra2006}, and recent studies demonstrated the practical efficiency of VarPro-inspired algorithms in modern deep-learning frameworks \cite{Newman2021}, existing analyses  do not provide rigorous convergence guarantees in regimes relevant to deep learning and PINNs.

Recent theoretical developments in the study of neural network training provide additional motivation for revisiting VarPro. Analyses of lazy training and kernel learning regimes have clarified when the feature map $\Phi(\theta)$ remains nearly fixed versus when it evolves substantially during optimization \cite{Chizat2019}. Complementarily, two-timescale analyses have shown that rapid adaptation of linear output weights coupled with slower feature evolution can lead to accelerated convergence and improved approximation properties \cite{Barboni2025}. From this perspective, VarPro can be interpreted as enforcing an instantaneous optimality condition on the linear variables, naturally aligning with fast-slow learning dynamics.

The present work builds on these insights to develop a convergence theory for VarPro algorithms coupled with deep neural networks. Our analysis connects classical SNLS results \cite{Golub1973,Ruhe1980} with modern perspectives on feature learning and over-parameterization, and clarifies how the conditioning of the feature map $\Phi(\theta)$ and the spectral properties of the reduced Hessian govern convergence behavior.

This paper is organized as follows. In Section \ref{sec:st_gr}, we recall the basics of Grassmannian geometry useful for our purposes. Then in Section \ref{sec:VarPro}, we study the VarPro optimization problem posed on the Grassmannian giving a fine analysis of the critical points. In Section \ref{sec:varpro_DNN}, the finite dimensional vector space is represented by a parametrized neural network. In such context, we define another optimization problem that couples neural networks and the VarPro optimization framework. In particular, we show that properties of critical points are still the same as in Section \ref{sec:VarPro} except when feature map $\Phi$ is not full rank. Then we analyze how to circumvent this rank deficiency in Section \ref{sec:stick_strata} where a regularization of the Grassmannian manifold is introduced. Finally in Section \ref{sec:num}, we show on some numerical examples that the method works in practice for regression or PINNs formulations. Furthermore, we give a scheme to solve the heat equation efficiently using the VarPro method and an implicit euler scheme. 

\section{The Stiefel and the Grassmannian manifolds}\label{sec:st_gr}

Here, we recall the basics of the geometry of finite dimensional vector spaces beginning by defining for $p \leq n \in \N$ the Stiefel manifold and the Grassmannian. Nothing is really new in sections \ref{sec:st_gr}-\ref{sec:VarPro} and all of the material can be found in the fundamental book \cite{Absil2008} sections 3-5 or the reference article \cite{Edelman1998}.

\begin{definition}
The Stiefel manifold, denoted $\St$, is the set of all orthonormal $p$-frames in $\R^n$ (where $p \le n$). It is defined as:
\begin{equation}
    \St = \{ X \in \R^{n \times p} \mid X^\top X = I_p \}
\end{equation}
\end{definition}

\subsection{The Tangent Space}

To work with this manifold, we first identify the tangent space at a point $X$.

\begin{lemma}[Characterization of the Stiefel Tangent Space]
The tangent space to $\St$ at a point $X$ is given by:
\begin{equation}
    T_X \St = \{ Z \in \R^{n \times p} \mid X^\top Z + Z^\top X = 0 \}
\end{equation}
Equivalently, $Z \in T_X \St$ if and only if $X^\top Z$ is a skew-symmetric matrix.
\end{lemma}

\begin{proof}
Let $\gamma : (-\epsilon, \epsilon) \to \St$ be a differentiable curve such that $\gamma(0) = X$ and $\dot{\gamma}(0) = Z$.
By definition, for all $t$, the curve satisfies the orthogonality constraint:
$$ \gamma(t)^\top \gamma(t) = I_p $$
Differentiating this expression with respect to $t$ and evaluating at $t=0$ yields:
\begin{align*}
    \frac{d}{dt} (\gamma(t)^\top \gamma(t)) \bigg|_{t=0} &= 0 \\
    \dot{\gamma}(0)^\top \gamma(0) + \gamma(0)^\top \dot{\gamma}(0) &= 0 \\
    Z^\top X + X^\top Z &= 0
\end{align*}
This implies that $X^\top Z$ is skew-symmetric, i.e., $X^\top Z \in \so(p)$.
\end{proof}
We can have an even more geometric interpretation of the tangent space with Proposition \ref{prop:tangent_so(n)} stating that the tangent space of $\St$ are the infinitesimal rotation of the columns of $X$.
\begin{proposition}\label{prop:tangent_so(n)}
Let $X \in \St$. The tangent space $T_X \St$ is given by:
\begin{equation}
    T_X \St = \{ ZX \mid Z \in \so(n) \}.
\end{equation}

\end{proposition}

\begin{proof}
The proof relies on the transitive action of the orthogonal group $O(n)$ on the Stiefel manifold.

\paragraph{Inclusion ($\supseteq$)}
Consider a matrix of the form $Y = ZX$ where $Z \in \so(n)$. We check if this $X$ satisfies the tangent space condition:
\begin{align*}
    X^\top Y + Y^\top X &= X^\top (ZX) + (ZX)^\top X \\
    &= X^\top Z X + X^\top Z^\top X \\
    &= X^\top (Z + Z^\top) X.
\end{align*}
Since $Z \in \so(n)$, we know that $Z^\top = -Z$ so $X^\top Y + Y^\top X = 0$
and $\{ ZX \mid Z \in \so(n) \} \subseteq T_X \St$.

\paragraph{Surjectivity via Dimensional Analysis}

Consider the linear map $\phi: \so(n) \to T_X \St$ defined by $\phi(Z) = ZX$. We must show that this map is surjective.

The dimension of the Stiefel manifold is:
\begin{equation}
    \dim(\St) = np - \frac{p(p+1)}{2}.
\end{equation}

The kernel of the map $\phi$ is the set of skew-symmetric matrices that vanish on the subspace spanned by $X$:
\begin{equation*}
    \ker(\phi) = \{ Z \in \so(n) \mid ZX = 0 \}.
\end{equation*}
Geometrically, this kernel corresponds to infinitesimal rotations that leave the columns of $X$ invariant. Since $X$ spans a $p$-dimensional subspace, these rotations act only on the orthogonal complement $X^\perp$, which is $(n-p)$-dimensional. Thus, the kernel is isomorphic to $\so(n-p)$.

By the Rank-Nullity Theorem applied to $\phi$:
\begin{align*}
    \dim(\text{Im}(\phi)) &= \dim(\so(n)) - \dim(\ker(\phi)) \\
    &= \dim(\so(n)) - \dim(\so(n-p)) \\
    &= \frac{n(n-1)}{2} - \frac{(n-p)(n-p-1)}{2}.
\end{align*}
Let us simplify this expression:
\begin{align*}
    \dim(\text{Im}(\phi)) &= \frac{1}{2} \left[ (n^2 - n) - ((n-p)^2 - (n-p)) \right] \\
    &= \frac{1}{2} \left[ n^2 - n - (n^2 - 2np + p^2 - n + p) \right] \\
    &= \frac{1}{2} \left[ 2np - p^2 - p \right] \\
    &= np - \frac{p(p+1)}{2}.
\end{align*}
We observe that $\dim(\text{Im}(\phi)) = \dim(T_X \St)$. Since $\text{Im}(\phi)$ is a subspace of $T_X \St$ with the same dimension, they must be equal. This finishes the proof of the Proposition.

\end{proof}

\subsection{The Grassmannian as a Quotient}

%The Grassmann manifold $\Gr$ is the set of all $p$-dimensional subspaces of $\R^n$. It can be viewed as a quotient of the Stiefel manifold. Since any basis of a subspace can be rotated into another orthonormal basis without changing the subspace, we have:
The Grassmann manifold $\Gr$ is the set of all $p$-dimensional subspaces of $\R^n$. It can be viewed as the quotient of the Stiefel manifold $\St$ by the orthogonal group $O(p)$. Since the subspace spanned by a set of basis vectors is invariant under the right action of any orthogonal matrix $Q \in O(p)$ (where $Q^\top Q = I_p$), we identify matrices in the Stiefel manifold that are equivalent up to this action:
\begin{equation}
    \Gr \cong \St / O(p)
\end{equation}
where $O(p)$ is the orthogonal group of size $p \times p$. Two matrices $X, Y \in \St$ represent the same point on the Grassmannian if $Y = XQ$ for some $Q \in O(p)$. 

One can identify the $\Gr$ as the space of orthogonal projections of rank $p$:

\begin{lemma}\label{lem:ident_Gr_proj}
$$\Gr \cong P_{n,p} := \{ P \in \R^{n\times n} : P^2 = P, P^{\top} = P, \text{rank}(P) = p \}$$
\end{lemma}

\begin{proof}
Let $\pi : \St \rightarrow P_{n,p}$ be defined by $\pi(X) = X X^\top$. For all $Q \in O(p)$, $\pi(X Q) = X Q Q^\top X^\top = X X^\top = \pi(X)$, so the map is invariant under the group action. We can therefore define the induced map on the quotient:
$$ \tilde{\pi} : \Gr \rightarrow P_{n,p}. $$
This map is bijective. 
\textbf{Injectivity:} Let $X, Y \in \St$ such that $\pi(X) = \pi(Y)$, i.e., $X X^\top = Y Y^\top$. This implies that $\text{Im}(X) = \text{Im}(Y)$, so $X$ and $Y$ span the same subspace and represent the same point in $\Gr$.
\textbf{Surjectivity:} Follows from the spectral theorem: any rank-$p$ projection matrix $P$ admits a decomposition $P=XX^\top$, where the columns of $X \in \R^{n \times p}$ form an orthonormal basis of $\text{Im}(P)$.
\end{proof}
Then, we endow $\St$ with the canonical Euclidean metric from the ambient space $\R^{n \times p}$:
\begin{equation}
    \langle A, B \rangle := \trace(A^\top B).
\end{equation}

To understand the first order structure of $\Gr$, one needs to characterize the vertical and the horizontal spaces of the fiber bundle $\St$ with base $\Gr$. 

\subsection{Vertical and Horizontal Decomposition}

Due to the quotient structure $\pi: \St \to P_{n,p} \cong \Gr$ with $\pi(X) = X X^{\top}$, we can decompose the tangent space $T_X \St$ into two orthogonal subspaces:
\begin{equation}
    T_X \St = \mathcal{V}_X \oplus \mathcal{H}_X
\end{equation}

\subsubsection{The Vertical Space}

The vertical space $\mathcal{V}_X$ is defined as the kernel of the differential of the projection map $\pi: \St \to \Gr$. Intuitively, these are the tangent directions at $X$ that do not cause any change in the subspace $\pi(X)$.

\begin{theorem}[Vertical Space Characterization]
The vertical space at $X$ is given by:
\begin{equation}
    \mathcal{V}_X = \{ X\Omega \mid \Omega \in \so(p) \}.
\end{equation}
\end{theorem}

\begin{proof}
Recall that the projection is given by $\pi(X) = XX^\top$. The vertical space is formally defined as the kernel of the differential:
$$ \mathcal{V}_X = \ker(d\pi(X)) = \{ Z \in T_X \St \mid d\pi(X)[Z] = 0 \}. $$
We compute the differential of $\pi$ at $X$ in the direction $Z$:
$$ d\pi(X)[Z] = Z X^\top + X Z^\top. $$

\paragraph{Inclusion ($\supseteq$)}
Let $Z = X\Omega$ with $\Omega \in \so(p)$. First, $Z \in T_X \St$ because $X^\top Z = \Omega$ is skew-symmetric. Substituting into the differential:
$$ d\pi(X)[X\Omega] = (X\Omega) X^\top + X (X\Omega)^\top = X (\Omega + \Omega^\top) X^\top. $$
Since $\Omega \in \so(p)$, $\Omega + \Omega^\top = 0$, so $d\pi(X)[Z] = 0$. Thus, $X\Omega \in \mathcal{V}_X$.

\paragraph{Inclusion ($\subseteq$)}
Let $Z \in \mathcal{V}_X$. By definition, $Z$ satisfies two conditions:
\begin{enumerate}
    \item $Z \in T_X \St \implies X^\top Z + Z^\top X = 0$ (skew-symmetry constraint).
    \item $d\pi(X)[Z] = 0 \implies Z X^\top + X Z^\top = 0$ (kernel constraint).
\end{enumerate}
Right-multiplying the kernel constraint equation by $X$:
$$ (Z X^\top + X Z^\top) X = 0 $$
$$ Z (X^\top X) + X (Z^\top X) = 0 $$
Since $X \in \St$, $X^\top X = I_p$. Also, from the tangent space condition, $Z^\top X = -X^\top Z$. Substituting these:
$$ Z I_p + X (-X^\top Z) = 0 $$
$$ Z = X (X^\top Z). $$
Let $\Omega := X^\top Z$. Since $Z \in T_X \St$, we know $\Omega$ is skew-symmetric ($\Omega \in \so(p)$).
The equation above becomes $Z = X\Omega$.
Thus, any vector in the vertical space is of the form $X\Omega$ with $\Omega \in \so(p)$.

This proves $\mathcal{V}_X = \{ X\Omega \mid \Omega \in \so(p) \}$.
\end{proof}

\subsubsection{The Horizontal Space}
The horizontal space $\mathcal{H}_X$ is defined as the orthogonal complement of $\mathcal{V}_X$ within $T_X \St$. This space provides a unique representation for tangent vectors of the Grassmannian.

\begin{theorem}[Horizontal Space Characterization]
The horizontal space at $X$ is given by:
\begin{equation}
    \mathcal{H}_X = \{ Z \in \R^{n \times p} \mid X^\top Z = 0 \}
\end{equation}
Note that the condition $X^\top Z = 0$ implies $X^\top Z + Z^\top X = 0$, so $Z$ is automatically in $T_X \St$.
\end{theorem}

\begin{proof}
A tangent vector $Z \in T_X \St$ is horizontal if and only if it is orthogonal to every vertical vector.
$$ Z \in \mathcal{H}_X \iff \forall V \in \mathcal{V}_X, \quad \langle Z, V \rangle = 0 $$
Using the characterization of $\mathcal{V}_X$, this condition becomes:
$$ \forall \Omega \in \so(p), \quad \trace(Z^\top X \Omega) = 0 $$
Let $A = X^\top Z$. Since $Z \in T_X \St$, we know from Lemma 2 that $A$ must be skew-symmetric ($A^\top = -A$).
The condition is rewritten as:
$$ \langle A, \Omega \rangle = 0 \quad \forall \Omega \in \so(p) $$
The only skew-symmetric matrix that is orthogonal to all skew-symmetric matrices is the zero matrix. Therefore, $A = 0$.
$$ X^\top Z = 0 $$
This proves that horizontal vectors are exactly those tangent vectors that are orthogonal to the columns of $X$.
\end{proof}

\begin{remark}\label{rem:ZopMMperp}
The matrix $Z \in \R^{n \times p}$ can be viewed as an operator from $M$ to $M^\perp$.
\end{remark}

\subsection{Characterization of the Grassmannian Tangent Space}

We can now rigorously characterize $T_{\mathcal{U}} \Gr$, where $\mathcal{U} := \text{span}(X)$. To do so, recall the definition of the projection operator $\pi : \St \rightarrow P_{n,p} \cong \Gr$, which we identify with the mapping to the projection matrix $\pi(X) = XX^\top.$

\begin{theorem}[Identification via Horizontal Lift]
There exists a linear isometry between the tangent space of the Grassmannian at $\pi(X)$ and the horizontal space at $X$:
\begin{equation}
    T_{\pi(X)} \Gr \cong \mathcal{H}_X = \{ Z \in \R^{n \times p} \mid X^\top Z = 0 \}
\end{equation}
\end{theorem}

\begin{proof}
By the Quotient Manifold Theorem, the Grassmannian $\Gr$ admits a smooth manifold structure such that the canonical projection $\pi : \St \to \Gr$ is a Riemannian submersion. 

As a submersion, the differential $d\pi(X) : T_X \St \to T_{\pi(X)} \Gr$ is a surjective linear map. We have previously established that the kernel of this map is exactly the vertical space $\mathcal{V}_X$. 
Since the horizontal space is defined as the orthogonal complement of the vertical space ($\mathcal{H}_X = \mathcal{V}_X^\perp$), the restriction of the differential to this subspace is injective. By the Rank-Nullity Theorem (or the First Isomorphism Theorem applied to the tangent bundle), a surjective linear map restricted to the orthogonal complement of its kernel is an isomorphism:
\begin{equation}
    d\pi(X) \big|_{\mathcal{H}_X} : \mathcal{H}_X \xrightarrow{\sim} T_{\pi(X)} \Gr.
\end{equation}
Consequently, for any tangent vector $\xi \in T_{\pi(X)} \Gr$, there exists a unique matrix $Z \in \mathcal{H}_X$, called the \textit{horizontal lift} of $\xi$, such that $d\pi(X)[Z] = \xi$.
\end{proof}

\begin{remark}
Using the horizontal characterization, we can easily compute the dimension of the manifold.

\begin{corollary}[Dimension]\label{cor:dimension_gr}
The dimension of $\Gr$ is $p(n-p)$.
\end{corollary}

\begin{proof}
The space $\mathcal{H}_X$ consists of matrices $Z \in \R^{n \times p}$ satisfying the linear constraint $X^\top Z = 0_{p \times p}$.
Since $X$ has rank $p$, the mapping $Z \mapsto X^\top Z$ imposes $p \times p = p^2$ independent linear constraints on the $np$ entries of $Z$.
\begin{align*}
    \dim(\Gr) &= \dim(\mathcal{H}_X) \\
    &= \dim(\R^{n \times p}) - \text{number of constraints} \\
    &= np - p^2 \\
    &= p(n-p)
\end{align*}
\end{proof}

\end{remark}

\begin{proposition}\label{prop:proj_tan}
Let $M \in \R^{n\times p}$, the orthogonal projection of $M$ on $T_{\pi(X)} \Gr$ is given by $(I_n - X X^{\top}) M $.
\end{proposition}

\begin{proof}
On the one hand as $X^{\top} X = I_p$, $X^{\top}(I_n - X X^{\top}) M = X^{\top} M - X^{\top} M = 0$ so $(I_n - X X^{\top}) M $ belongs to $T_{\pi(X)} \Gr$. 
On the other hand for all $Z \in T_{\pi(X)} \Gr$,  $Tr(Z^{\top} X X^{\top} M) = Tr(M^{\top} X X^{\top} Z) = 0$.
\end{proof}

A global picture of the geometry is given in Figure \ref{fig:synth_geom}.

\begin{figure}
\centering
\begin{tikzpicture}[
    scale=1.2,
    >={Latex[length=3mm, width=2mm]},
    font=\small
]

    % --- 1. ESPACE DE BASE (La variété M / Grassmannienne) ---
    \begin{scope}[yshift=-4cm]
        % Surface courbe (représentation locale du manifold)
        \fill[gray!10] (-3,0) to[out=20,in=160] (3,0) to[out=-20,in=200] (3,-2) to[out=160,in=20] (-3,-2) -- cycle;
        \draw[gray!60, thick] (-3,0) to[out=20,in=160] (3,0); % Bord arrière
        \draw[black, thick] (-3,0) to[out=-20,in=200] (-3,-2) to[out=20,in=160] (3,-2) to[out=200,in=-20] (3,0); % Bords
        
        % Point projeté U
        \coordinate (U) at (0, -1);
        \fill[black] (U) circle (2pt);
        \node[below right] at (U) {$\mathcal{U} = \pi(X)$};
        \node[gray, above right] at (2.5, -0.5) {Base $\Gr$};
        
        % Vecteur tangent sur la base
        \draw[->, blue, thick] (U) -- (1.5, -0.8) node[right] {$d\pi(X)_{|\mathcal{H}_X}$};
    \end{scope}

    % --- 2. LA FIBRE (Ligne courbe verticale) ---
    \draw[thick, dashed, gray] (0, -4) .. controls (-0.5, -2) and (0.5, 0) .. (0, 2);
    \node[left, gray] at (-0.3, 1.5) {Fiber $\pi^{-1}(\mathcal{U})$};

    % --- 3. ESPACE TOTAL (Autour de X) ---
    \coordinate (X) at (0, 0); 
    \fill[black] (X) circle (2pt) node[above left] {$X$};

    % --- 4. ESPACES TANGENTS ---
    
    % A. Espace Vertical (Tangent à la fibre)
    \draw[red, thick, ->] (X) -- (0.2, 1.2) node[above] {$\mathcal{V}_X$};
    \draw[red, thick, ->] (X) -- (-0.15, -0.9);
    
    % B. Espace Horizontal (Orthogonal à la fibre)
    \fill[blue!20, opacity=0.6] (-1.5, -0.5) -- (1.5, -0.5) -- (2.0, 0.8) -- (-1.0, 0.8) -- cycle;
    \draw[blue!50] (-1.5, -0.5) -- (1.5, -0.5) -- (2.0, 0.8) -- (-1.0, 0.8) -- cycle;
    
    % Vecteur Horizontal
    \draw[blue, thick, ->] (X) -- (1.5, 0.2) node[right] {$\mathcal{H}_X$};

    % --- 5. FLÈCHE DE PROJECTION ---
    \draw[->, thick, dashed] (0.5, -0.5) to[out=-60, in=60] node[midway, right] {$\pi$} (0.2, -3.8);

    % --- 6. ANNOTATIONS ---
    \node[right] at (2.5, 1) {Total space $\St$};
    
    % Symbole orthogonalité
    \draw[thin] (0.1, 0.4) -- (0.3, 0.35) -- (0.2, -0.05);

\end{tikzpicture}
\caption{The fiber bundle $\St$}
\label{fig:synth_geom}
\end{figure}
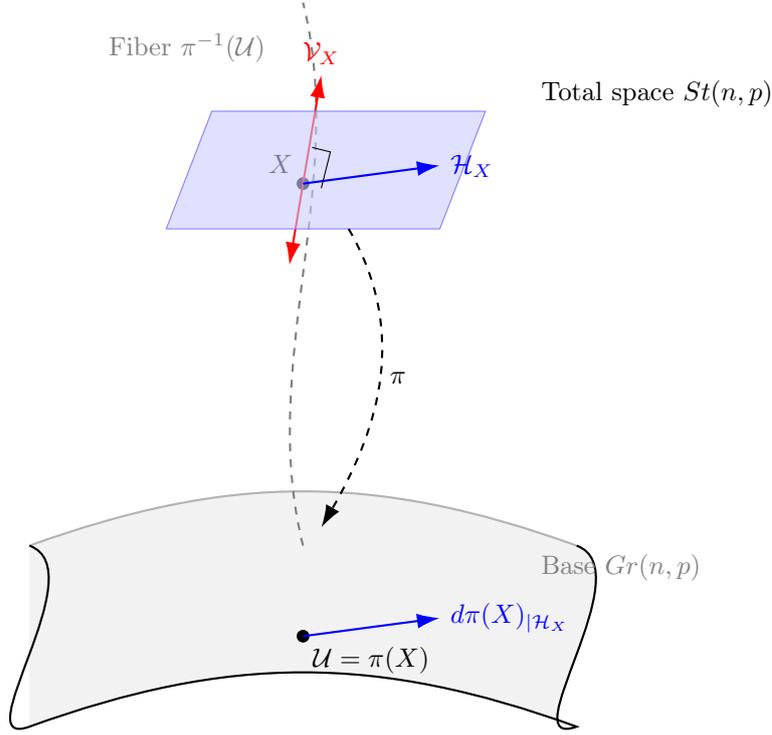
\subsection{Geodesics and Second-Order Approximation}

In this section, we provide a rigorous characterization of the geodesics on $\Gr$ and derive their second-order Taylor expansion. This expansion is critical for analyzing the convergence of Newton-like algorithms, as it reveals the local curvature of the manifold.

\subsubsection{Geodesics via Horizontal Lifts}

Since the Grassmannian is a quotient manifold of the Stiefel manifold ($\Gr \cong \St / O(p)$), the geometry of $\Gr$ is best understood through the theory of Riemannian submersions. A fundamental result \cite[Proposition 2.109]{gallot2004riemannian} states that geodesics on the base manifold $\Gr$ are the projections of specific geodesics on the total space $\St$, known as \textbf{horizontal geodesics}.

\begin{definition}
A curve $Y(t)$ in $\St$ is an horizontal geodesic if it satisfies two conditions:
\begin{enumerate}
    \item It is a geodesic with respect to the canonical metric on $\St$ (i.e., zero covariant acceleration).
    \item Its velocity vector is horizontal everywhere: $\dot{Y}(t) \in \mathcal{H}_{Y(t)}$ for all $t$.
\end{enumerate}
\end{definition}
The projection $\pi(Y(t))$ is then a geodesic on $\Gr$. The following theorem gives the closed-form expression for this lift.

\begin{theorem}[Horizontal Geodesic in $\St$]
Let $X \in \St$ be a base point and $H \in \mathcal{H}_X$ be a horizontal tangent vector (satisfying $X^\top H = 0$). Let the compact Singular Value Decomposition (SVD) of $H$ be:
$$H = U \Sigma V^\top$$
where $U \in \R^{n \times p}$ has orthonormal columns (and $U^\top X = 0$), $\Sigma = \diag(\sigma_1, \dots, \sigma_p) \succeq 0$, and $V \in O(p)$ is orthogonal.

The curve $Y: \R \to \R^{n \times p}$ defined by:
\begin{equation} \label{eq:geodesic}
    Y(t) = X V \cos(t\Sigma) V^\top + U \sin(t\Sigma) V^\top
\end{equation}
is the unique horizontal geodesic in $\St$ passing through $X$ at $t=0$ with velocity $\dot{Y}(0) = H$.
Consequently, the curve $\gamma(t) = \text{span}(Y(t))$ is the geodesic on $\Gr$ starting at $\pi(X)$ with velocity represented by $H$.
\end{theorem}

\begin{proof}
We need to verify the following three conditions:
\begin{enumerate}
    \item $Y(t) \in \St$ for all $t$ (feasibility).
    \item $\dot{Y}(t) \in \mathcal{H}_{Y(t)}$ for all $t$ (horizontality).
    \item $\nabla^{\St}_{\dot{Y}} \dot{Y} = 0$ (zero covariant acceleration).
\end{enumerate}

\paragraph{1. Feasibility ($Y(t) \in \St$).}
We verify $Y(t)^\top Y(t) = I_p$. Note that $X^\top H = 0$ implies $X^\top U \Sigma V^\top = 0$, and thus $X^\top U = 0$.
\begin{align*}
    Y(t)^\top Y(t) &= \left( V \cos(t\Sigma) V^\top X^\top + V \sin(t\Sigma) U^\top \right) \left( X V \cos(t\Sigma) V^\top + U \sin(t\Sigma) V^\top \right) \\
    &= V \cos(t\Sigma) \underbrace{V^\top X^\top X V}_{I_p} \cos(t\Sigma) V^\top 
     + V \sin(t\Sigma) \underbrace{U^\top U}_{I_p} \sin(t\Sigma) V^\top \\
    &= V (\cos^2(t\Sigma) + \sin^2(t\Sigma)) V^\top = V I_p V^\top = I_p.
\end{align*}

\paragraph{2. Horizontality ($Y(t)^\top \dot{Y}(t) = 0$).}
Differentiating $Y(t)$, we get:
$$ \dot{Y}(t) = -X V \Sigma \sin(t\Sigma) V^\top + U \Sigma \cos(t\Sigma) V^\top. $$
We compute the inner product:
\begin{align*}
    Y^\top \dot{Y} &= \left( V \cos(t\Sigma) V^\top X^\top + V \sin(t\Sigma) U^\top \right) 
                      \left( -X V \Sigma \sin(t\Sigma) V^\top + U \Sigma \cos(t\Sigma) V^\top \right) \\
                   &= - V \cos(t\Sigma) \Sigma \sin(t\Sigma) V^\top + V \sin(t\Sigma) \Sigma \cos(t\Sigma) V^\top.
\end{align*}
Since diagonal matrices commute, the two terms cancel out, so $Y(t)^\top \dot{Y}(t) = 0$.

\paragraph{3. Zero Covariant Acceleration.}
We compute the Euclidean acceleration $\ddot{Y}(t)$:
\begin{align*}
    \ddot{Y}(t) &= -X V \Sigma^2 \cos(t\Sigma) V^\top - U \Sigma^2 \sin(t\Sigma) V^\top \\
    &= - \left( X V \cos(t\Sigma) V^\top + U \sin(t\Sigma) V^\top \right) V \Sigma^2 V^\top \\
    &= - Y(t) (V \Sigma^2 V^\top).
\end{align*}
The matrix $S := V \Sigma^2 V^\top$ is symmetric. The normal space to $\St$ (in $\R^{n \times p}$) at $Y$ is $\mathcal{N}_Y \St := \{ Y S \mid S = S^\top \}$ (recall that the tangent space of $\St$ at $Y$ is $\so(n) Y$). Thus, $\ddot{Y}(t)$ lies entirely in the normal space. The covariant acceleration is the projection of $\ddot{Y}(t)$ onto the tangent space, which is therefore zero.
\end{proof}

\subsubsection{Second-Order Taylor Expansion}

We now derive the second-order approximation of the geodesic to understand the local deviation from the tangent space.

\begin{proposition}[Retraction at Order 2]
The Taylor expansion of the horizontal geodesic $Y(t)$ around $t=0$ is:
\begin{equation}
    Y(t) = X + tH - \frac{t^2}{2} X H^\top H + O(t^3)
\end{equation}
where $H:= \dot{Y}(0)$.
\end{proposition}

\begin{proof}
We use the expression of the derivatives computed above at $t=0$.
For the velocity:
$$ \dot{Y}(0) = -0 + U \Sigma I_p V^\top = U \Sigma V^\top = H. $$
For the acceleration:
$$ \ddot{Y}(0) = -Y(0) V \Sigma^2 V^\top = -X (V \Sigma V^\top) (V \Sigma V^\top). $$
Recall that $H = U \Sigma V^\top$, so $H^\top H = V \Sigma U^\top U \Sigma V^\top = V \Sigma^2 V^\top$.
Substituting this into the acceleration term yields:
$$ \ddot{Y}(0) = -X (H^\top H). $$
By Taylor's theorem, $Y(t) = Y(0) + t\dot{Y}(0) + \frac{t^2}{2}\ddot{Y}(0) + O(t^3)$, which gives the result.
\end{proof}

\begin{remark}
The term $- \frac{1}{2} X H^\top H$ corresponds to the normal component of the acceleration required to keep the curve on the manifold. A mapping $R_X(tH)$ that agrees with this expansion up to order 2 is called a \textbf{second-order retraction}.
\end{remark}

\section{The Variable Projection Functional}\label{sec:VarPro}

We now turn our attention to the separable nonlinear least squares problem. Let $y \in \R^n$ be a target vector. We assume our model function depends on two sets of parameters: linear parameters $\alpha \in \R^p$ and nonlinear parameters that define a basis matrix.

\subsection{Problem Formulation}

The standard separable least squares problem is given by:
\begin{equation} \label{eq:standard_ls}
    \min_{X \in \St, \alpha \in \R^p} \frac{1}{2} \| y - X \alpha \|_2^2
\end{equation}
where the matrix $X$ represents the nonlinear dependence. In the context of the VarPro algorithm, we eliminate the linear parameters $\alpha$ by observing that for a fixed $X$, the optimal $\hat{\alpha}$ is the solution to a linear least squares problem:
\begin{equation}
    \hat{\alpha}(X) = X^\dagger y = (X^\top X)^{-1} X^\top y = X^\top y
\end{equation}
where the simplification holds because $X \in \St$. Substituting this back into \eqref{eq:standard_ls}, we obtain the reduced cost function depending only on $X$:
\begin{equation}
    f(X) = \frac{1}{2} \| y - X X^\top y \|_2^2 = \frac{1}{2} \| (I_n - P_\mathcal{U}) y \|_2^2 = \frac{1}{2} (\| y \|_2^2 - \|P_\mathcal{U} y \|_2^2)
\end{equation}
where $P_\mathcal{U} := XX^\top$ is the orthogonal projector onto the subspace spanned by $X$ denoted $\mathcal{U}$.

This is equivalent to solving the following problem where we use the isomorphism given by Lemma \ref{lem:ident_Gr_proj}:
\begin{equation}\label{eq:varpro_cost}
    \min_{\mathcal{U} \in \Gr} \frac{1}{2} \| (I - P_\mathcal{U}) y \|_2^2 = \min_{\mathcal{X} \in \St} \frac{1}{2} \|(I_n - X X^{\top}) y \|_2^2.
\end{equation}

\subsection{Riemannian Gradient}

To apply gradient descent methods on the manifold to minimize the cost function $F(\mathcal{U}) = - \frac{1}{2} \| P_\mathcal{U} y \|_2^2$, we require the Riemannian gradient. We define the representative function on the Stiefel manifold as:
\begin{equation}
    \tilde{F}(X) = -\frac{1}{2} \trace(y^\top X X^\top y) = -\frac{1}{2} \trace(X^\top y y^\top X).
\end{equation}

\begin{proposition}[Euclidean Gradient]
The gradient of $\tilde{F}$ in the ambient Euclidean space $\R^{n \times p}$ is:
\begin{equation}
    \nabla \tilde{F}(X) = - y y^\top X.
\end{equation}
\end{proposition}

\begin{proof}
Let $Z \in \R^{n \times p}$ be a perturbation. The directional derivative is:
\begin{align*}
    D\tilde{F}(X)[Z] &= -\frac{1}{2} \trace(Z^\top y y^\top X + X^\top y y^\top Z) \\
    &= -\trace(Z^\top y y^\top X) \\
    &= \langle -y y^\top X, Z \rangle.
\end{align*}
\end{proof}

To obtain the Riemannian gradient on the Grassmannian, we project the Euclidean gradient onto the horizontal space $\mathcal{H}_X = \{ Z \in \R^{n \times p} \mid X^\top Z = 0 \}$.

\begin{theorem}[Riemannian Gradient Formula]\cite[Section 3.6]{Absil2008}
  The Riemannian gradient of the cost function, denoted $\mathrm{grad} F(X)$, is the horizontal lift of the gradient on the manifold. It is given by:
\begin{equation} \label{eq:riemannian_grad}
    \mathrm{grad} F(X) = (I_n - XX^\top) \nabla \tilde{F}(X) = -(I_n - XX^\top) y y^\top X.
\end{equation}
\end{theorem}

\begin{proof}
By Proposition \ref{prop:proj_tan}, the orthogonal projection of a matrix $A$ onto $\mathcal{H}_X$ is given by $P_{\mathcal{H}_X}(A) = (I_n - XX^\top) A$. Applying this to $\nabla \tilde{F}(X)$:
$$
\mathrm{grad} F(X) = (I_n - XX^\top) (-y y^\top X) 
$$

\end{proof}

\subsection{Riemannian Hessian Analysis}

To analyze the local convergence of gradient descent methods, we must compute the Riemannian Hessian. The Riemannian Hessian at $X$ along a direction $H \in \mathcal{H}_X$ is given by:
\begin{equation}
    \mathrm{Hess} F(X)[H] = P_{\mathcal{H}_X} \left( D(\mathrm{grad} F(X))[H] \right),
\end{equation}
where $D(\cdot)[H]$ denotes the classical Euclidean directional derivative.

\begin{theorem}[Riemannian Hessian Formula]\cite[Section 5.5]{Absil2008}
Let $X \in \St$ and $H \in \mathcal{H}_X$. Let $M = y y^\top$. The Riemannian Hessian of the VarPro objective on the manifold $\Gr$ is:
\begin{equation} \label{eq:hessian_formula}
    \mathrm{Hess} F(X)[H] = -(I_n - XX^\top) M H + H (X^\top M X).
\end{equation}
\end{theorem}

\begin{proof}
We differentiate the gradient vector field $\mathrm{grad} F(X) = - (I_n - XX^\top) M X$.
Let us expand the expression:
$$ G(X) = - M X + X X^\top M X. $$
We compute the directional derivative $D(G(X))[H]$ term by term:
\begin{enumerate}
    \item $D(-MX)[H] = -MH$.
    \item $D(X X^\top M X)[H] = H (X^\top M X) + X (H^\top M X) + X (X^\top M H)$.
\end{enumerate}
Summing these gives the Euclidean derivative of the vector field:
$$ D(G(X))[H] = -MH + H (X^\top M X) + X H^\top M X + X X^\top M H. $$
Now, we apply the projection $P_{\mathcal{H}_X} = (I_n - XX^\top)$ to restrict this to the horizontal space.
\begin{itemize}
    \item \textbf{Term 1:} $P_{\mathcal{H}_X}(-MH) = -(I_n - XX^\top) MH$.
    \item \textbf{Term 2:} Recall $H$ is horizontal ($X^\top H = 0$). Thus $(I_n - XX^\top)H = H$.
    $$ P_{\mathcal{H}_X}(H (X^\top M X)) = H (X^\top M X). $$
    \item \textbf{Term 3:} $P_{\mathcal{H}_X}(X (H^\top M X)) = (I_n - XX^\top) X (\dots) = 0$.
    \item \textbf{Term 4:} $P_{\mathcal{H}_X}(X (X^\top M H)) = (I_n - XX^\top) X (\dots) = 0$.
\end{itemize}
Combining the non-zero terms yields the result:
$$ \mathrm{Hess} F(X)[H] = -(I_n - XX^\top) M H + H (X^\top M X). $$
\end{proof}

\subsection{Global Convergence Analysis}

A key property of the Variable Projection functional on the Grassmannian is the absence of spurious local minima. The landscape is benign: all critical points that are not global minima are saddle global maxima, allowing gradient-based methods to converge to the global solution almost surely.

\begin{theorem}[Landscape of Critical Points]\label{th:landscape}
    Consider the cost function $f: \Gr \to \R$ defined by:
    \begin{equation}
        f(\mathcal{U}) = \frac{1}{2} \| (I_n - P_{\mathcal{U}}) y \|_2^2,
    \end{equation}
    where $P_{\mathcal{U}} \in P_{n,p}$ is the orthogonal projection matrix onto the subspace $\mathcal{U}$. 
    
    Let $\mathcal{C} = \{ \mathcal{U} \in \Gr \mid \mathrm{grad} f(\mathcal{U}) = 0 \}$ be the set of critical points. For any $\mathcal{U} \in \mathcal{C}$, exactly one of the following holds:
    \begin{enumerate}
        \item $y \in \mathcal{U}$. In this case, $\mathcal{U}$ is a global minimum with $f(\mathcal{U}) = 0$.
        \item $y \perp \mathcal{U}$. In this case, $\mathcal{U}$ is a global maximum with $f(\mathcal{U}) = \frac{1}{2}\|y\|_2^2$.
    \end{enumerate}
    Furthermore, the Riemannian Hessian at any global maximum admits a direction of strictly negative curvature (provided $p < n$). Consequently, the function has no spurious local minima: every local minimum is a global minimum.
\end{theorem}

\begin{proof}
To analyze the critical points, we work with the horizontal lift of the problem on the Stiefel manifold. Let $X \in \St$ be a representative such that $\mathcal{U} = \Span(X)$. The cost function becomes $F(X) = \frac{1}{2} \| (I_n - XX^\top) y \|_2^2$.
    
    Since the projection $\pi: \St \to \Gr$ is a Riemannian submersion, the critical points and the Hessian properties of $f$ on $\Gr$ correspond exactly to those of $F$ on $\St$ when restricted to the horizontal space $\mathcal{H}_X$.

    \textbf{1. Classification of Critical Points} \\
    Recall from Equation \eqref{eq:riemannian_grad} that the Riemannian gradient is given by:
    $$ \mathrm{grad} F(X) = -(I_n - XX^\top) y y^\top X. $$
    Let $r := (I_n - XX^\top) y \in \R^n$ be the residual vector, and $v = X^\top y \in \R^p$ be the coefficients of the projection. The gradient can be written as the outer product:
    $$ \mathrm{grad} F(X) = - r v^\top. $$
    For the gradient to vanish, this matrix must be zero. Since it is a rank-1 (or rank-0) matrix, it vanishes if and only if one of the vectors is zero:
    \begin{itemize}
        \item \textbf{Case A:} $r = 0$. This implies $(I_n - XX^\top) y = 0$, meaning $y = XX^\top y$. Thus, $y \in \Span(X)$. The cost function attains its lower bound $F(X)=0$, making this a \textbf{global minimum}.
        \item \textbf{Case B:} $v = 0$. This implies $X^\top y = 0$, meaning $y \perp \Span(X)$. In this case, $P_X y = 0$ and the cost is maximal $F(X) = \frac{1}{2}\|y\|^2$. We must verify the second-order condition to classify it further.
    \end{itemize}

\textbf{2. Hessian Analysis for Case B} \\
    Consider a critical point $X$ where $y \perp \Span(X)$. We examine the Riemannian Hessian derived in Equation \eqref{eq:hessian_formula}:
    $$ \mathrm{Hess} F(X)[H] = -(I_n - XX^\top) y y^\top H + H (X^\top y y^\top X). $$
    Since $X^\top y = 0$, the second term vanishes ($X^\top y y^\top X = 0$). The expression simplifies to:
    $$ \mathrm{Hess} F(X)[H] = -(I_n - XX^\top) y y^\top H. $$
    Furthermore, since $y$ is orthogonal to $X$, $(I_n - XX^\top)y = y$. The Hessian action becomes:
    $$ \mathrm{Hess} F(X)[H] = - y y^\top H. $$
    We now compute the quadratic form of the Hessian along a specific tangent direction. We construct a descent direction $H$ that rotates the subspace towards $y$. Let $u \in \R^p$ be an arbitrary unit vector. Define:
    $$ H = y u^\top. $$
    First, we verify that $H$ is a valid horizontal tangent vector ($H \in \mathcal{H}_X$):
    $$ X^\top H = X^\top (y u^\top) = (X^\top y) u^\top = 0 \cdot u^\top = 0. $$
    Now, we evaluate the Hessian quadratic form $\langle H, \mathrm{Hess} F(X)[H] \rangle$:
    \begin{align*}
        \langle H, \mathrm{Hess} F(X)[H] \rangle &= \trace(H^\top (-y y^\top H)) \\
        &= - \trace( (u y^\top) y y^\top (y u^\top) ) \\
        &= - \trace( u (y^\top y) (y^\top y) u^\top ) \\
        &= - \|y\|_2^4 \cdot \|u\|_2^2 \\
        &< 0 \quad (\text{assuming } y \neq 0).
    \end{align*}
    Since the Hessian admits a strictly negative eigenvalue, the critical point $X$ is unstable (a strict local maximum in this specific direction). Therefore, no point satisfying $y \perp \Span(X)$ can be a local minimum.

\end{proof}

\section{The VarPro algorithm coupled with neural networks}\label{sec:varpro_DNN}

Let $\Theta$ be a compact parameter space and $(x_1, \cdots, x_n)$ being a set of points sampled from a compact domain $\Omega$. We introduce what we call a \textbf{feature map} $\Phi: \Theta \rightarrow \R^{n \times p}$ and denote:

$$
\forall \theta \in \Theta, \ X(\theta) = \Phi(\theta,x_i) \in \R^{n \times p}.
$$
In what follows, we will often use the following notation for matrices $Z \in \R^{n \times p}$:

$$
\Span(Z):= \Span \{ (Z_{ij})_{i \leq n} \in \R^n : 1 \leq j \leq  p\}
$$ 
and for sets of matrices $S$:

$$
\Span(S):= \bigcup_{Z \in S} \Span(Z).
$$ 
\subsection{When the feature map $X$ is non degenerate}
In this section in order to stick with the Grassmannian framework, we assume that that $X$ is not degenerate in the sense that $\rank(X) =p$. Without loss of generality, we suppose that the columns of $X$ are orthogonal. Hence, it is possible to define:

$$
\mathcal{U}(\theta) := \Span X(\theta).
$$

The space $\mathcal{U}(\theta)$ can be generated by any network architecture: shallow, deep, convolutionnal, recurrent...

\subsubsection{The overparametrized regime}

Moreover, one states the fundamental assumption of over parametrization:
\begin{assumption}[Over parametrization]\label{as:rankp}
For all $\theta \in \Theta$, 
$$
\mathcal{U}(\theta)^\perp \subseteq \Span(\Image(d X (\theta)))  
$$.
\end{assumption}
This hypothesis corresponds to the over parametrization regime as in order to be satisfied the dimension of the parameter space should be at least $n-p$. The main theorem of the section is the following:

\begin{theorem}[Landscape of Critical Points]\label{th:landscape_G}
    Consider the cost function $G: \Theta \to \R$ defined by:
    \begin{equation}
        G(\theta) = \frac{1}{2} \| (I_n - X(\theta) X(\theta)^\top) y \|_2^2.
    \end{equation}
    Let $\mathcal{C} = \{ \theta \in \Theta \mid \mathrm{grad} G(\theta) = 0 \}$ be the set of critical points. For any $\theta \in \mathcal{C}$ and under assumption \ref{as:rankp}, exactly one of the following holds:
    \begin{enumerate}
        \item $y \in \mathcal{U}(\theta)$. In this case, $\theta$ is a global minimum with $G(\theta) = 0$.
        \item $y \perp \mathcal{U}(\theta)$. In this case, $\theta$ is a global maximum.
    \end{enumerate}
    Consequently, every local minimum of $G$ is a global minimum.
\end{theorem}

\begin{remark}
We lost the saddle property of a global maximum. To recover it, the map $X$ should be a submersion which is a strong hypothesis.
\end{remark}

\begin{proof}
The function $G$ writes:

$$
G = F \circ X
$$
so that:

$$
\begin{array}{rl}
dG(\theta)(d\theta) =& d F(X(\theta))(d X(\theta) d\theta)\\
=& \trace((d X(\theta) d\theta)^\top (I_n - XX^\top) y y^\top X)
\end{array}
$$
Thus for any $\theta \in C$ and by the Assumption \ref{as:rankp}, columns of
$
(I_n - XX^\top) y y^\top X
$
are perpendicular to $\mathcal{U}(\theta)^\perp$. As $(I_n - XX^\top) y y^\top X \in \mathcal{H}_X$, its columns are also perpendicular to $\mathcal{U}(\theta)$ so:
\begin{equation}\label{eq:grad_F_theta}
\mathrm{grad}F(X(\theta)) =   (I_n - XX^\top) y y^\top X = 0.
\end{equation}
As in the proof of Theorem \ref{th:landscape}, either $y \in \mathcal{U}(\theta)$ and $\theta$ is a global minimum or $y \in \mathcal{U}(\theta)^\perp$ and is a global maximum. 
%In such case,
%
%$$
%\begin{array}{rl}
%\mathrm{Hess} G(\theta)[d\theta] =& \mathrm{Hess} F(X(\theta))[dX(\theta) d\theta] + dF(X(\theta))(\mathrm{Hess} X(\theta) [d\theta])\\
%=&\mathrm{Hess} F(X(\theta))[dX(\theta) d\theta].
%\end{array}
%$$
%Again by Assumption \ref{as:rankp} and the proof of Theorem \ref{th:landscape}, $\theta \in C$ admits a direction of negative curvature.
\end{proof}

%We have even more:
%\begin{corollary}
%For almost every initialization $\theta_0 \in \Theta$, the gradient flow converges to the set of global optima.
%\end{corollary}
%
%\begin{proof}
%Since $\Theta$ is compact, trajectories converge to the set of critical points. The critical points consist of saddle points and global minima (by Theorem \ref{th:landscape_G}).
%By assumption, the saddle points are hyperbolic. This implies two crucial properties:
%\begin{enumerate}
%    \item Saddle points are isolated, so there are finitely many of them in the compact set $\Theta$.
%    \item The Local Stable Manifold Theorem \cite[Theorem 1.3.2]{GuckenheimerHolmes1983} applies: for each saddle $p$, the local stable manifold is a submanifold of codimension at least one.
%\end{enumerate}
%The global basin of attraction of a saddle $p$, $W^s(p)$, is the backwards propagation of the local manifold along the flow. Since the flow is a diffeomorphism, $W^s(p)$ retains codimension at least one and thus has Lebesgue measure zero.
%Consequently, the union of the basins of attraction of all saddle points is a finite union of measure zero sets, which is of measure zero. Almost all trajectories must therefore converge to the set of remaining critical points, which are global minima.
%\end{proof}

\subsubsection{Comment on Assumption \ref{as:rankp}}
In fact if $\dim(\Theta)$ is large enough, Assumption \ref{as:rankp} is verified almost surely by Theorem \ref{th:almost_sure_as}. 

\begin{theorem}\label{th:almost_sure_as}
Suppose that the $(x_i)_{i\leq n}$ are sampled uniformly and independently in $\Omega$ and $\Phi$ analytic. For all $\theta \in \Theta$ and denoting by $\zeta_k$ the local charts of $\Theta$ around $\theta$ then Assumption \ref{as:rankp} is satisfied almost surely if:

\begin{itemize}

\item $p (1+\dim(\Theta)) \geq n$, 
\item the family $x \rightarrow \Phi(\theta,x), \partial_{\zeta_k} \Phi(\theta,x)$ is independent in the vector space of functions mapping $\Omega$ to $\mathbb{R}$.
\end{itemize}
\end{theorem}

\begin{proof}
The space  $ \mathcal{W}(\theta) := \mathcal{U}(\theta) + \Span(\Image(d X(\theta)))$ is generated by $p (1+\dim(\Theta)) $ vectors of $\R^n$. Moreover, $\mathcal{U}(\theta)$ is generated by the columns of $\Phi(\theta,x_i)$ whereas $\Span(\Image(dX(\theta))$ is generated by columns of $\partial_{\zeta_k} \Phi(\theta,x_i) $. As the family $x \rightarrow \Phi(\theta,x), \partial_{\zeta_k} \Phi(\theta,x)$ is independent in the vector space of functions mapping $\Omega$ to $\mathbb{R}$ and by Theorems \ref{th:rank}, the space $\mathcal{W}(\theta)$ is of dimension $n$ almost surely so that Assumption \ref{as:rankp} is verified. 
\end{proof}

\begin{remark}
In the case of shallow neural network, the feature map reads:

$$
\forall j \leq p, \ \Phi(\theta,x)_j = \sigma(w_j \cdot x + b_j)
$$ 
where we used the under-script $j$ to designate the $j$th column. The derivative with respect to the parameters are:
$$
\begin{array}{rl}
\partial_{w_j} \Phi(\theta,x)_j =& \sigma^\prime(w_j \cdot x + b_j) x \\ 
\partial_{b_j} \Phi(\theta,x)_j =& \sigma^\prime(w_j \cdot x + b_j).
\end{array}
$$
For the ReLU activation, the condition of independence is satisfied if:
\begin{itemize}
\item For $k\neq j$ $(w_j,b_j) \neq (w_k, b_k)$.
\item For all $j \leq p$, $\text{Supp}(x \rightarrow \Phi(\theta,x)_j) \cap \Omega \neq \emptyset$.
\end{itemize} 
However the ReLU activation is not analytic so one cannot apply directly Theorem \ref{th:rank}.

\begin{proof}[\textbf{Sketch of Proof: Independence across distinct neurons}]
Let $\mathcal{F} = \bigcup_{j=1}^p \{ \sigma_j, \sigma'_j, x \sigma'_j \}$ where we denote $\sigma_j(x) = \sigma(w_j \cdot x + b_j)$.
Assume a linear combination vanishes on $\Omega$:
$$
\sum_{j=1}^p \left( c_j \sigma_j(x) + d_j \sigma'_j(x) + \langle e_j, x \rangle \sigma'_j(x) \right) = 0, \quad \forall x \in \Omega
$$
We aim to show that for any distinct $j$, the associated term must vanish individually.

\textbf{1. Geometric Separation:}
Let $H_j = \{ x \in \mathbb{R}^d \mid w_j \cdot x + b_j = 0 \}$ be the defining hyperplane for neuron $j$.
Since $(w_j, b_j)$ are distinct pairs, $H_j \neq H_k$ for $j \neq k$. Thus, there exists an open set $S_j \subset H_j \cap \Omega$ such that for any $x \in S_j$, $x \notin H_k$ for all $k \neq j$.

\textbf{2. Local Regularity Argument:}
Consider a point $x_0 \in S_j$.
\begin{itemize}
    \item For all $k \neq j$, the functions $\sigma_k$ and $\sigma'_k$ are $C^\infty$ (specifically, locally affine or zero) in a neighborhood of $x_0$.
    \item For the index $j$, $\sigma_j$ is $C^0$ but not $C^1$ (the "kink"), and $\sigma'_j$ is discontinuous (the Heaviside "jump") across $H_j$.
\end{itemize}

\textbf{3. Isolating the Discontinuity:}
Evaluating the limit of the equation as we approach $x_0$ from opposite sides of $H_j$:
$$
\lim_{\epsilon \to 0^+} \text{Eq}(x_0 + \epsilon w_j) - \lim_{\epsilon \to 0^-} \text{Eq}(x_0 + \epsilon w_j) = 0.
$$
The smooth terms ($k \neq j$) and the continuous term $c_j \sigma_j$ cancel out. The only remaining terms are the discontinuous $\sigma'_j$ components:
$$
\text{Jump} = (d_j + \langle e_j, x_0 \rangle) \cdot (1 - 0) = 0.
$$
This implies $d_j + \langle e_j, x \rangle = 0$ for all $x \in S_j$. Since $S_j$ is a subset of the hyperplane, this constrains the affine form to vanish on $H_j$.

\textbf{4. Isolating the Kink:}
With the jump terms neutralized, we look at the second derivative (distributional sense). The term $c_j \sigma_j$ has a second derivative proportional to a Dirac delta $\delta(w_j \cdot x + b_j)$, while all $k \neq j$ terms have zero second derivative near $x_0$.
For the sum to be zero, the coefficient of the Dirac delta must be zero:
$$
c_j = 0.
$$

\end{proof}
\end{remark}

%Here is the proposition that ensures the submersion Assumption \ref{as:rankp}.
%\begin{proposition}\label{prop:submersion}
%Assume that for all $\theta \in \Theta$, the following rank condition holds:
%\[
%\mathcal{U}(\theta) + \Span(\operatorname{Im}(dX(\theta))) = \R^n
%\]
%then $\mathcal{U}$ is a Riemannian submersion.
%\end{proposition}
%
%\begin{proof}
%To prove that $\mathcal{U}$ is a  submersion, we must check the surjectivity of the differential $d\mathcal{U}$.
%
%Recall how the function $\mathcal{U}$ is built:
%
%$$
%\theta \rightarrow \Phi(\theta,x_i) \in \R^{n \times p} \xrightarrow{\text{Gram-Schmidt}} X \in \St \xrightarrow{\pi} \mathcal{U} \in \Gr
%$$
%so that the differential reads:
%$$
%d \theta \rightarrow d \Phi(\theta,x_i)[d\theta] \in \R^{n \times p} \xrightarrow{\text{Gram-Schmidt}} dX(\theta)[d\theta] \in T_X \St \xrightarrow{d \pi} d\mathcal{U}(\theta)[d\theta] \in T_\mathcal{U} \Gr
%$$
%where the Gram-Schmidt arrow designates the linear transformation that computes the orthonormalization of the columns of $X$. By assumption $\mathcal{U}(\theta)^\perp \subset \Span(\Image(d X(\theta))) $ the horizontal space at $X(\theta)$. By definition of $T_\mathcal{U} \Gr$, $\Image(d \mathcal{U}(\theta)) = \Image(d \pi \circ dX ) = d \pi( \mathcal{H}_X)   = T_\mathcal{U} \Gr$ giving the result.
%\end{proof}
Now we prove Theorem \ref{th:rank} that ensures the almost everywhere independence.
\begin{theorem}[Proved in \cite{huang2006extreme}]\label{th:rank}
Let $\mathcal{F} = \{f_1, \dots, f_p\}$ be a family of $p$ real analytic functions on $\Omega$. 
Assume that the family $\mathcal{F}$ is linearly independent in the vector space of functions mapping $\Omega$ to $\mathbb{R}$.
Let $X_1, \dots, X_p$ be independent and identically distributed (i.i.d.) random variables sampled according to Lebesgue measure on $\Omega$.
Then, the sampling matrix $M \in \mathbb{R}^{p \times p}$ defined by:
\[
M_{ij} = f_j(X_i)
\]
is invertible almost surely. That is:
\[
\mathbb{P}\left( \det(M) \neq 0 \right) = 1.
\]
\end{theorem}

\begin{proof}
We proceed by induction on the size of the family, $p$.

\paragraph{Base Case ($p=1$):}
The matrix $M$ is a scalar $M_{11} = f_1(X_1)$.
Since the family $\{f_1\}$ is linearly independent, $f_1$ is not the zero function. As $f_1$ is analytic, its zero set has measure zero. Thus:
\[
\mathbb{P}(\det(M) = 0) = \mathbb{P}(f_1(X_1) = 0) = 0.
\]
The base case holds.

\paragraph{Inductive Step:}
Assume the theorem holds for any set of $p-1$ linearly independent analytic functions.
Let $\mathcal{F} = \{f_1, \dots, f_p\}$ be a linearly independent family. Consider the matrix $M = (f_j(X_i))_{1 \le i,j \le p}$.

We expand the determinant along the last row (corresponding to the sample $X_p$):
\begin{equation} \label{eq:laplace}
\det(M) = \sum_{j=1}^p (-1)^{p+j} f_j(X_p) \cdot \Delta_{p,j}(X_1, \dots, X_{p-1})
\end{equation}
where $\Delta_{p,j}$ is the $(p,j)$-minor of the matrix. Note that the term $\Delta_{p,p}$ corresponds exactly to the determinant of the sub-matrix formed by the functions $\{f_1, \dots, f_{p-1}\}$ evaluated at points $\{X_1, \dots, X_{p-1}\}$.

Let us condition on the variables $X_1, \dots, X_{p-1}$. Define the function $\Psi : \Omega \to \mathbb{R}$ (random with respect to the first $p-1$ samples) as:
\[
\Psi(x) = \sum_{j=1}^p C_j \cdot f_j(x)
\]
where the coefficients are given by $C_j = (-1)^{p+j} \Delta_{p,j}(X_1, \dots, X_{p-1})$.

By the inductive hypothesis applied to the sub-family $\{f_1, \dots, f_{p-1}\}$, we know that:
\[
\forall j \leq p, \ \mathbb{P}(C_j \neq 0) = 1.
\]
Consider a fixed realization $(x_1, \dots, x_{p-1})$ such that at least one $C_j \neq 0$. The function $\Psi$ is a linear combination of the functions $f_j$. Since the family $\mathcal{F}$ is linearly independent and the coefficients $C_j$ are not all zero, the function $\Psi$ cannot be identically zero.

Furthermore, since $\Psi$ is a linear combination of analytic functions, it is itself analytic. Thus, its zero set has measure zero. We obtain:
\[
\mathbb{P}(\Psi(X_p) = 0 \mid X_1, \dots, X_{p-1}) = 0.
\]

\paragraph{Conclusion:}
Almost surely, $\det(M) \neq 0$. 
\end{proof}

If the features functions are not analytic, Theorem \ref{th:rank_pl_corrected} gives general conditions to get the same conclusion as Theorem \ref{th:rank}.

\begin{theorem}[Almost-sure invertibility under an absolute-continuity condition]
\label{th:rank_pl_corrected}
Let $\Omega\subset\mathbb R^d$ be open with positive Lebesgue measure.
Let $\mathcal F=\{f_1,\dots,f_p\}$ be real-valued measurable functions on $\Omega$
and define the evaluation map
\[
\Psi:\Omega\to\mathbb R^p,\qquad \Psi(x)=(f_1(x),\dots,f_p(x)).
\]

Assume the following:
\begin{enumerate}
  \item \textbf{Linear independence:} $f_1,\dots,f_p$ are linearly independent
  as functions on $\Omega$ to $\R$.
  \item \textbf{Absolute-continuity of the pushforward:} The pushforward measure
  $\mu:=\Psi_*(\lambda|_\Omega)$ (where $\lambda$ denotes Lebesgue measure)
  is absolutely continuous with respect to Lebesgue measure on $\mathbb R^p$
  (i.e. $\mu$ has a density).
\end{enumerate}

Let $X_1,\dots,X_p$ be i.i.d.\ with law $\lambda|_\Omega$ and define the
sampling matrix $M\in\mathbb R^{p\times p}$ by
\[
 M_{ij}=f_j(X_i).
\]
Then
\[
\mathbb P\bigl(\det M\neq 0\bigr)=1.
\]
\end{theorem}

\begin{proof}
Let $Y_i:=\Psi(X_i)\in\mathbb R^p$. By definition of the pushforward,
$Y_1,\dots,Y_p$ are i.i.d.\ with common law $\mu$, and by assumption $\mu$
is absolutely continuous w.r.t.\ Lebesgue measure on $\mathbb R^p$; hence
$\mu$ has a density and in particular assigns measure zero to any Lebesgue
-null set in $\mathbb R^p$.

Consider the map
\[
\det:(\mathbb R^p)^p\to\mathbb R,\qquad
\det(y^{(1)},\dots,y^{(p)})=\det\bigl( y^{(i)}_j\bigr)_{1\le i,j\le p},
\]
the determinant of the $p\times p$ matrix whose \(i\)-th row is \(y^{(i)}\).
The set
\[
\mathcal Z := \{(y^{(1)},\dots,y^{(p)})\in(\mathbb R^p)^p:\det(y^{(1)},\dots,y^{(p)})=0\}
\]
is a real algebraic hypersurface in \((\mathbb R^p)^p\) (it is the zero set
of the nontrivial polynomial \(\det\)), hence it has Lebesgue measure zero
in \((\mathbb R^p)^p\).

Since $\mu$ has a density, the product measure $\mu^{\otimes p}$ on
\((\mathbb R^p)^p\) is absolutely continuous with respect to Lebesgue
measure on \((\mathbb R^p)^p\). Therefore $\mu^{\otimes p}(\mathcal Z)=0$.
But
\[
\mathbb P(\det M=0)=\mathbb P\bigl(\det(Y_1,\dots,Y_p)=0\bigr)
= \mu^{\otimes p}(\mathcal Z)=0,
\]
and the result follows.
\end{proof}
The hypothesis of Theorem \ref{th:rank_pl_corrected} are rather abstract. We apply it to piecewise linear functions in Corollary \ref{cor:rank_pl} to show its usefulness.
\begin{corollary}[Almost-sure invertibility for piecewise linear functions]
\label{cor:rank_pl}

Let $\Omega$ be an open set with positive Lebesgue measure.
Let $\mathcal F=\{f_1,\dots,f_p\}$ be real-valued piecewise linear functions on
$\Omega$. Assume that $\Omega$ can be decomposed as a union  polyhedral cells $U \subset \Omega$ such that:
\begin{enumerate}
  \item Each $f_j$ is affine on $U$, i.e for all $x \in U$.
  \[
  f_j(x)=a_j\cdot x+b_j,\qquad j=1,\dots,p;
  \]
  \item The vectors $a_1,\dots,a_p\in\mathbb R^d$ are linearly independent.
\end{enumerate}

Let $X_1,\dots,X_p$ be i.i.d.\ random variables uniformly distributed on
$\Omega$, and define the sampling matrix
\[
M_{ij}=f_j(X_i).
\]
Then
\[
\mathbb P\bigl(\det M\neq 0\bigr)=1.
\]
\end{corollary}

\begin{proof}
Define the evaluation map
\[
\Psi:\Omega\to\mathbb R^p,\qquad
\Psi(x)=(f_1(x),\dots,f_p(x)).
\]

On a polyhedral cell $U$, the map $\Psi$ is affine and has the form
\[
\Psi(x)=Ax+b,
\]
where $A\in\mathbb R^{p\times d}$ has rows $a_1,\dots,a_p$.
By assumption, $\mathrm{rank}(A)=p$, hence $\Psi|_U$ is an open map from $U$
onto the open set $A(U)\subset\mathbb R^p$.

Therefore, the pushforward of Lebesgue measure on $U$ by $\Psi|_U$ is absolutely
continuous with respect to Lebesgue measure on $\mathbb R^p$, with a density
bounded below on $A(U)$. Since $U$ has positive measure, the pushforward of
Lebesgue measure on $\Omega$ by $\Psi$ is also absolutely continuous.

\end{proof}

%the calculus of the hessian gives:
%$$
%\begin{array}{rl}
%\langle d\theta, \mathrm{Hess} G(\theta)[d\theta] \rangle =& - \trace( (dX d\theta)^\top(I_n - XX^\top) y y^\top (d X d\theta)) \\
%&= - \trace( (dX d\theta)^\top y y^\top (d X d\theta))\\
%&= - (y^{\top} (dX d\theta))^2.
%\end{array}
%$$
%So either $y \perp \Image(dX(\theta))$ and we cannot say anything or we have a saddle point.

\subsubsection{The under parametrization regime}
When the Assumption \ref{as:rankp} is not satisfied, the following situation holds:
\begin{theorem}\label{th:underparametrization}
Let $\theta$ be a critical point of $G$. Suppose the following injectivity condition holds for $dX$: there exists $C>0$ such that: 
$$
\forall z \in \R^p, \forall d\theta \in T_\theta \Theta, \ |d X(\theta)[d\theta] z | \geq C |z| |d\theta|.
$$
Then exactly one of the following holds:
\begin{enumerate}
    \item $(I_n - XX^\top) y \perp \Span(\Image(d X(\theta)))$. This is a local minimum if the residual is small relative to the curvature, specifically if:
    $$ \| (I_n - XX^\top) y \|_2 < C^2 \frac{\| X^\top y \|_2}{\| d^2 X \|_{op}}. $$
    Otherwise, $\theta$ may be a saddle point.
    \item $y \perp \mathcal{U}(\theta)$. In this case, $\theta$ is a global maximum with $G(\theta) = \frac{1}{2} \|y\|^2$.
\end{enumerate}
\end{theorem}

\begin{proof}
Recall that the gradient of the composite function is given by:
$$
dG(\theta)[d\theta] = \langle \mathrm{grad} F(X), dX(\theta)[d\theta] \rangle = - \trace( dX(\theta)[d\theta]^\top (I_n - XX^\top) y y^\top X ).
$$
Let $r = (I_n - XX^\top) y$ be the residual and $v = X^\top y$ be the projection coefficients. The condition $dG(\theta) = 0$ implies either:
\begin{enumerate}
    \item $v = 0$, which implies $y \perp \mathcal{U}(\theta)$ (Case 2, Global Maximum).
    \item $r \perp \Span(\Image(d X(\theta)))$, meaning $r^\top dX(\theta)[d\theta] = 0$ for all $d\theta$ (Case 1).
\end{enumerate}

In \textbf{Case 1}, we analyze the Hessian quadratic form $Q(d\theta) := \langle d\theta, \mathrm{Hess} G(\theta)[d\theta] \rangle$. Using the chain rule, this decomposes into the Riemannian Hessian of $F$ and the curvature term from $X$:
$$
Q(d\theta) = \underbrace{\langle dX(\theta)[d\theta], \mathrm{Hess} F(X)[dX(\theta)[d\theta]] \rangle}_{\text{Term 1}} + \underbrace{\langle \mathrm{grad} F(X), d^2 X[d\theta, d\theta] \rangle}_{\text{Term 2}}.
$$

For \textbf{Term 1}, we use the Riemannian Hessian formula with $H = dX(\theta)[d\theta]$:
$$
\begin{array}{rl}
\mathrm{Hess} F(X)[H] =& -(I_n - XX^\top) y y^\top H + H (X^\top y y^\top X) \\
=& - r y^\top H + H v v^\top.
\end{array}
$$
The inner product is:
$$
\begin{array}{rl}
\langle H, \mathrm{Hess} F(X)[H] \rangle =& -\trace(H^\top r y^\top H) + \trace(H^\top H v v^\top) \\
=& - (y^\top H)(H^\top r) + \|H v\|_2^2.
\end{array}
$$
Since we are in Case 1, $r \perp \Span(\Image(dX))$, which implies $H^\top r = 0$. Thus, the first part vanishes, leaving only:
$$ \text{Term 1} = \| dX(\theta)[d\theta] v \|_2^2. $$

For \textbf{Term 2}, we substitute $\mathrm{grad} F(X) = - r v^\top$:
$$
\text{Term 2} = \langle - r v^\top, d^2 X[d\theta, d\theta] \rangle = - \langle r, d^2 X[d\theta, d\theta] v \rangle.
$$

Combining both terms and applying the injectivity hypothesis $|dX(\theta)[d\theta]v| \ge C|v||d\theta|$:
$$
\begin{array}{rl}
Q(d\theta) =& \| dX(\theta)[d\theta] v \|_2^2 - \langle r, d^2 X[d\theta, d\theta] v \rangle \\
\geq & C^2 \|v\|^2 |d\theta|^2 - \|r\| \|d^2 X\|_{op} \|v\| |d\theta|^2 \\
=& \left( C^2 - \frac{\|r\| \|d^2 X\|_{op}}{\|v\|} \right) \|v\|^2 |d\theta|^2.
\end{array}
$$
Thus, if the residual $\|r\|$ is small enough, $Q(d\theta) > 0$ for all $d\theta \neq 0$, proving $\theta$ is a local minimum.
\end{proof}

\begin{remark}
The injectivity assumption is verified as soon as the family $x \rightarrow \partial_{\zeta_k} \Phi(\theta,x)$ is independent in the space of function from $\Omega$ to $\R$ by Theorems \ref{th:rank}-\ref{th:rank_pl_corrected}.
\end{remark}

\subsection{When the feature map $X$ is degenerate}

The exact description of the function $\theta \rightarrow P_{\mathcal{U}(\theta)}$ used in a learning algorithm is the following. First, we build the features as before:
$$
\forall \theta \in \Theta, \ Y(\theta) = \Phi(\theta,x) \in \mathbb{R}^{n \times p}.
$$
Then we use a Gram-Schmidt algorithm to produce a matrix of orthonormal vectors. Let $r := \rank(Y(\theta))$. The scheme generates the projector $P_{\mathcal{U}(\theta)}$ as follows:

\begin{equation}\label{eq:scheme}
\theta 
\xrightarrow{\Phi} Y(\theta) \in \mathbb{R}^{n \times p} 
\xrightarrow{\text{Gram-Schmidt}} X(\theta)  
\rightarrow P_{\mathcal{U}(\theta)} = (XX^\top)(\theta)
\end{equation}

\begin{proposition}\label{prop:continuity_strata}
The mapping $\theta \mapsto X(\theta)$ is continuous on the strata of constant rank of $Y(\theta)$, but discontinuous at points where the rank changes.
\end{proposition}
\begin{proof}
Let $\Theta_k = \{ \theta \in \Theta \mid \rank(Y(\theta)) = k \}$ denote the stratum of rank $k$.
\newline

\textbf{Continuity on $\Theta_k$:} Restricting $\theta$ to a stratum $\Theta_k$, the Gram-Schmidt process involves a fixed sequence of elementary algebraic operations (additions, multiplications) and normalizations. The normalization steps involve division by the norm of intermediate vectors. Since the rank is constant and equal to $k$, the selected vectors are linearly independent, guaranteeing that these norms are strictly non-zero. As a composition of continuous functions on a domain avoiding singularities, the resulting projector $P_{\mathcal{U}(\theta)}$ varies continuously with $\theta$.
\newline 

\textbf{Discontinuity across strata:} Consider the Trace operator, which is a continuous function on the space of matrices $\mathbb{R}^{n \times n}$. For any orthogonal projector $P$, $\text{Tr}(P) = \rank(P)$. Therefore, for any $\theta \in \Theta_k$, we have $\text{Tr}(P_{\mathcal{U}(\theta)}) = k$.
Let $(\theta_j)_{j \in \mathbb{N}}$ be a sequence in $\Theta_k$ converging to $\theta^* \in \Theta_{k'}$ with $k \neq k'$. If the mapping $X$ were continuous at $\theta^*$, we would have:
$$ \lim_{j \to \infty} \text{Tr}(P_{\mathcal{U}(\theta_j)}) = \text{Tr}(P_{\mathcal{U}(\theta^\star)}) \implies \lim_{j \to \infty} k = k' $$
This is a contradiction since $k$ and $k'$ are distinct integers. Thus, $P_{\mathcal{U}(\theta)}$ is discontinuous at any point where the rank of the feature matrix $Y(\theta)$ changes.
\end{proof}

Proposition \ref{prop:continuity_strata} is fundamental since it states that the VarPro cost function $\eqref{eq:varpro_cost}$  is smooth only on on each stratum $\Theta_k$ separately. This makes the definition of a gradient descent across strata difficult. In Section \ref{sec:stick_strata}, we propose to circumvent this difficulty by ``filling'' the space in between strata.

\subsection{Applications to machine learning problems}
For numerous applications in neural network learning algorithms, the VarPro method can be applied:
\begin{itemize}

\item The most obvious application is for regression problems. Consider a function $f: \Omega \rightarrow \R$, then defining $y := (f(x_i))_{1\leq i \leq n}$, the optimization problem reads:

$$
\min_{\theta \in \Theta, \alpha \in \R^p} \frac{1}{n} \sum_{i=1}^n (f(x_i) - \alpha^{\top} Y(\theta,x_i))^2 
$$
is equivalent to:

$$
\min_{\theta \in \Theta} \| y - P_{\mathcal{U}(\theta)} y\|_n^2 
$$
where $\|z\|^2_n = \frac{1}{n} \sum_{i=1}^n z_i^2$.

\item Also PINN formulations correspond to the VarPro framework. Indeed, let $\mathcal{L} : H \rightarrow L^2(\Omega)$ where $H \subset L^2(\Omega)$ be a PDE operator like the Laplacian for example. Also let $\mathcal{B}$ be a boundary operator. Then let $[x_1, \cdots, x_n]$ be sampled in the domain and $[b_{n+1}, \cdots,b_{m+n}]$ be sampled on the boundary. Then build the matrix $A$: 
$$
\left\{
\begin{array}{rl}
A_{i,j} = \mathcal{L} \Phi(\theta_j,x_i) \text{ if } i \leq n \\
A_{i,j} = \mathcal{B} \Phi(\theta_j,b_i) \text{ if } i > n 
\end{array}
\right.
$$
and the vector $v$:
$$
\left\{
\begin{array}{rl}
v_{i} = f(x_i) \text{ if } i \leq n \\
v_{i} = g(b_i) \text{ if } i > n. 
\end{array}
\right.
$$
The PINNs problem is equivalent to minimize:

$$
\min_{\theta, \alpha} G(\theta, \alpha) := \| A(\theta) \alpha - v  \|^2 = \| Q(\theta) R(\theta) \alpha - v  \|^2 .
$$
where $A =: Q R$ where $Q$ is orthonormal and $R$ is upper triangular and invertible if $A$ is full rank. In such case, this is equivalent to:

$$
\min_{\theta, \alpha} \|Q(\theta) \alpha - v  \|^2,
$$
a least square problem of the usual form \eqref{eq:standard_ls}.

\end{itemize}

%For Ritz formulation things get more complicated. Indeed, take a bilinear symmetric form on $H$ and $l$ a linear form on $L^2(\Omega)$. Then the Ritz problem is equivalent to minimizing : $u \rightarrow \frac{1}{2} a(u,u) - l(u)$. The problem is that usually in the bilinear form $a$ some space derivative are hidden and not in the linear form $l$ so we have two different type of samples $(A\Phi(\theta_j,x_i), \Phi(\theta_j,x_i))$ where $A$ is a differential operator. This makes the VarPro approach impossible here as it is.

\section{Geometry and Optimization on the Regularized Projector Manifold} \label{sec:stick_strata}

In this section, we introduce a specific manifold of regularized projectors. This study allows to continuously pass from each stratum $Gr(n,k)$ for $k\leq p$ by introducing a regularization parameter $\varepsilon > 0$.

\subsection{The Regularized Projector Manifold}

We consider the space of matrices generated by a smooth mapping from the Euclidean space $\R^{n \times p}$ to the space of symmetric matrices $\Sym(n)$.

\begin{definition}[Regularized Projector Map]
Let $n, p \in \N$ with $p \le n$ and let $\varepsilon > 0$. The regularized projection map $\pi_\varepsilon : \R^{n \times p} \to \Sym(n)$ is defined by:
\begin{equation}
    \pi_\varepsilon(X) = X (X^\top X + \varepsilon I_p)^{-1} X^\top
\end{equation}
We define the manifold $\mathcal{M}_\varepsilon := \Image(\pi_\varepsilon)$ as the image of this map.
\end{definition}

For brevity, we define the kernel matrix $K(X) := (X^\top X + \varepsilon I_p)^{-1}$. Thus, a point on the manifold is given by $P = X K X^\top$.

\begin{proposition}[Differential of $\pi_\varepsilon$]
Let $X \in \R^{n \times p}$ and $\Delta \in \R^{n \times p}$ be a tangent vector in the parameter space. Let $K = (X^\top X + \varepsilon I_p)^{-1}$ and $P = \pi_\varepsilon(X)$. The Fréchet derivative of $\pi_\varepsilon$ at $X$ in the direction $\Delta$ is given by:
\begin{equation}
    D\pi_\varepsilon(X)[\Delta] = (I_n - P) \Delta K X^\top + \left( (I_n - P) \Delta K X^\top \right)^\top.
\end{equation}
\end{proposition}

\begin{proof}
Let $M(X) = X^\top X + \varepsilon I_p$, so that $K = M(X)^{-1}$. We differentiate the expression $P = X K X^\top$ using the product rule:
\begin{equation*}
    D\pi_\varepsilon(X)[\Delta] = \Delta K X^\top + X (DK[\Delta]) X^\top + X K \Delta^\top.
\end{equation*}
Using the matrix calculus identity $D(M^{-1}) = -M^{-1} (DM) M^{-1}$, we compute the differential of the kernel matrix:
\begin{align*}
    DK[\Delta] &= -K (D(X^\top X)[\Delta]) K \\
               &= -K (\Delta^\top X + X^\top \Delta) K.
\end{align*}
Substituting this back into the expression for $D\pi_\varepsilon$:
\begin{align*}
    D\pi_\varepsilon(X)[\Delta] &= \Delta K X^\top - X \left( K \Delta^\top X K + K X^\top \Delta K \right) X^\top + X K \Delta^\top \\
    &= \Delta K X^\top - (X K \Delta^\top) (X K X^\top) - (X K X^\top) (\Delta K X^\top) + X K \Delta^\top.
\end{align*}
Recall that $P = X K X^\top$. We can group the terms as follows to highlight the projection structure:
\begin{align*}
    D\pi_\varepsilon(X)[\Delta] &= \underbrace{\Delta K X^\top - P (\Delta K X^\top)}_{\text{Terms 1 and 3}} + \underbrace{X K \Delta^\top - (X K \Delta^\top) P}_{\text{Terms 4 and 2}} \\
    &= (I_n - P) \Delta K X^\top + \left( (I_n - P) \Delta K X^\top \right)^\top.
\end{align*}
This confirms the symmetric structure of the differential.
\end{proof}

\subsection{Spectral Analysis and Boundary Behavior}

In this section, we analyze the spectral properties of the matrices in $\mathcal{M}_\varepsilon$. This analysis justifies the structure of the manifold's closure relevant to the optimization problem.

\begin{proposition}[Spectral Characterization]
Let $X \in \R^{n \times p}$. The regularized projector $P = \pi_\varepsilon(X)$ admits the eigendecomposition:
\begin{equation}
    P = U \Lambda U^\top
\end{equation}
where $\Lambda = \diag(\lambda_1, \dots, \lambda_p)$ with $0\leq \lambda_i < 1$ and $U\in \St$
\end{proposition}

\begin{proof}
We introduce the SVD of $X$ \textit{i.e} $X = U \Sigma V$ with $U \in St(n,p), V \in O(p)$ and $\Sigma = \diag(\sigma_1, \cdots, \sigma_p)$.
We substitute the SVD of $X$ into the definition of $\pi_\varepsilon$. Recall that $X^\top X = V \Sigma^2 V^\top$. Then:
\begin{align*}
    K(X) &= (V \Sigma^2 V^\top + \varepsilon I_p)^{-1} \\
         &= (V (\Sigma^2 + \varepsilon I_p) V^\top)^{-1} \\
         &= V (\Sigma^2 + \varepsilon I_p)^{-1} V^\top.
\end{align*}
Now substituting $K(X)$ into the expression for $P$:
\begin{align*}
    P &= X K(X) X^\top \\
      &= (U \Sigma V^\top) \left( V (\Sigma^2 + \varepsilon I_p)^{-1} V^\top \right) (V \Sigma U^\top) \\
      &= U \Sigma (\Sigma^2 + \varepsilon I_p)^{-1} \Sigma U^\top \\
      &= U \diag\left( \frac{\sigma_1^2}{\sigma_1^2 + \varepsilon}, \dots, \frac{\sigma_p^2}{\sigma_p^2 + \varepsilon} \right) U^\top.
\end{align*}
Since $\varepsilon > 0$ and $\sigma_i \in \R$, we have $0 \le \sigma_i^2 < \sigma_i^2 + \varepsilon$, which implies $0 \le \lambda_i < 1$.
\end{proof}

\subsection{The Spectral Manifold Definition}

Rather than parameterizing the manifold via matrices in Euclidean space, we define it intrinsically as a collection of symmetric matrices with bounded eigenvalues.

\begin{definition}[Spectral Manifold $\mathcal{M}$]
Let $n, p \in \N$ with $p \le n$. We define the regularized projector manifold $\mathcal{M}$ as the set of symmetric matrices of rank at most $p$ whose non-zero eigenvalues lie strictly in the interval $[0, 1)$.
\begin{equation}
    \mathcal{M} := \left\{ U \Lambda U^\top \in \Sym(n) \mid U \in \St, \ \Lambda = \diag(\lambda_1, \dots, \lambda_p), \ 0 \le \lambda_i < 1 \right\}
\end{equation}
\end{definition}

This manifold bridges the zero matrix and the set of orthogonal projectors. Its closure, $\overline{\mathcal{M}}$, is obtained by allowing the eigenvalues to take the value $1$.
\begin{equation}
    \overline{\mathcal{M}} = \left\{ U \Lambda U^\top \mid U \in \St, \ 0 \le \Lambda \le I_p \right\}.
\end{equation}
The boundary where $\Lambda = I_p$ corresponds exactly to the Grassmannian $\Gr$, represented by orthogonal projectors.

\subsubsection{Quotient Geometry and Topology}

The manifold structure is best understood as a quotient space. Since the spectral definition fixes the frame alignment, we generalize it by parameterizing the manifold using arbitrary symmetric matrices within the subspace. Consider the product:
\[
\mathcal{T} = \St \times \mathcal{S}_p^{(0,1)},
\]
where $\mathcal{S}_p^{(0,1)}$ is the convex set of $p \times p$ symmetric matrices with eigenvalues in $[0, 1)$. The group $O(p)$ acts on both factors: on the frame $U \in \St$ by right multiplication, and on the matrix $S \in \mathcal{S}_p^{(0,1)}$ by conjugation. We define the equivalence relation:
\begin{equation}
    (U, S) \sim (U Q, Q^\top S Q), \quad \forall Q \in O(p).
\end{equation}
The map $\Psi(U, S) = U S U^\top$ is invariant under this action and induces a diffeomorphism between the quotient space and $\mathcal{M}$.

Topologically, this construction identifies $\mathcal{M}$ as the \textbf{associated vector bundle}:
\begin{equation}
    \mathcal{M} \cong \St \times_{O(p)} \mathcal{S}_p^{(0,1)}.
\end{equation}
This is a bundle \emph{over} the Grassmannian $\Gr$. The construction via $\St$ captures the non-trivial ``twisting'' of the bundle: a rotation of the basis within the subspace (a change in $U$) is compensated by a counter-rotation of the intrinsic operator $S$.

\subsection{Tangent Space Characterization}

We characterize the tangent space $T_P \mathcal{M}$ by varying the spectral components. A tangent vector results from a variation in the eigenvalues (changing the regularization strength) and a rotation of the eigenvectors (changing the subspace).

\begin{theorem}[Tangent Space Structure]
Let $P = U \Lambda U^\top \in \mathcal{M}$ with distinct eigenvalues. The tangent space at $P$ is the direct sum of diagonal variations and commutator-based rotations:
\begin{equation}
    T_P \mathcal{M} = \left\{ U \dot{\Lambda} U^\top + [Z, P] \ \mid \ \dot{\Lambda} \in \mathcal{D}_p, \ Z \in \mathfrak{so}(n) \right\}
\end{equation}
where $\mathcal{D}_p$ is the space of diagonal $p \times p$ matrices and $\mathfrak{so}(n)$ is the Lie algebra of skew-symmetric matrices.
\end{theorem}

\begin{proof}
Consider a curve $P(t) = U(t) \Lambda(t) U(t)^\top$. Differentiating at $t=0$:
\[
\dot{P} = \dot{U} \Lambda U^\top + U \dot{\Lambda} U^\top + U \Lambda \dot{U}^\top.
\]
Since $U \in \St$, the variation $\dot{U}$ can be expressed as $\dot{U} = Z U$ for some skew-symmetric matrix $Z \in \mathfrak{so}(n)$ by Proposition \ref{prop:tangent_so(n)}. Substituting $\dot{U} = Z U$:
\begin{align*}
    \dot{P} &= Z U \Lambda U^\top + U \dot{\Lambda} U^\top + U \Lambda (Z U)^\top \\
            &= Z (U \Lambda U^\top) + U \dot{\Lambda} U^\top + (U \Lambda U^\top) Z^\top \\
            &= Z P + U \dot{\Lambda} U^\top - P Z \quad (\text{since } Z^\top = -Z) \\
            &= [Z, P] + U \dot{\Lambda} U^\top.
\end{align*}
The term $[Z, P]$ captures the motion along the Grassmannian (rotating the subspace), while $U \dot{\Lambda} U^\top$ captures the motion orthogonal to the Grassmannian fibers (scaling the eigenvalues).
\end{proof}

For optimization algorithms, it is fundamental to express the tangent vector at some $P \in \mathcal{M}$ with respect to the parameter space with variable $X \in \R^{n\times p}$. This is done in Corollary \ref{cor:exp_tangentspace_parameter}.
\begin{corollary}[Components of the Differential]\label{cor:exp_tangentspace_parameter}
Let $X = U \Sigma V^\top$ be the SVD of $X$. Consider a tangent vector $\Delta \in \R^{n \times p}$ decomposed as $\Delta = U \Omega_A V^\top + U_\perp \Omega_B V^\top$.
The differential $\eta = D\pi_\varepsilon(X)[\Delta]$ corresponds to a tangent vector $\nu = U \dot{\Lambda} U^\top + [Z, P]$ where:
\begin{enumerate}
    \item The eigenvalue variation $\dot{\Lambda}$ (diagonal) is driven by the diagonal of $\Omega_A$:
    \begin{equation}
        \dot{\lambda}_i = \frac{2}{\sqrt{\varepsilon}} \sqrt{\lambda_i} (1 - \lambda_i)^{3/2} \ (\Omega_A)_{ii}.
    \end{equation}
    \item The subspace rotation $Z$ (antisymmetric) has components:
    \begin{itemize}
        \item External rotation ($Z_{21}$):
        \begin{equation}
            Z_{21} = \frac{1}{\sqrt{\varepsilon}} \Omega_B \ \diag\left( \sqrt{\frac{1-\lambda_i}{\lambda_i}} \right).
        \end{equation}
        \item Internal rotation ($Z_{11}$): Depends on the off-diagonal parts of $\Omega_A$ (where $\lambda_i \neq \lambda_j$).
    \end{itemize}
Note that we used the notation:

$$
\forall 1 \leq i \leq p, \ \lambda_i := \frac{\sigma_i^2}{\sigma_i^2 + \varepsilon}.
$$    
    
\end{enumerate}
\end{corollary}

\begin{proof}
We identify the components of the differential $\eta = D\pi_\varepsilon(X)[\Delta]$ by matching them against the general tangent structure derived in the Theorem.

\paragraph{Step 1: Block Structure of the Tangent Vector.}
Let $\mathcal{B} = [U, U_\perp]$ be the basis aligned with the SVD of $X$. In this frame, $P = \operatorname{diag}(\Lambda, 0)$. For a general tangent vector $\nu = U \dot{\Lambda} U^\top + [Z, P]$, we decompose the skew-symmetric matrix $Z$ as $\begin{psmallmatrix} Z_{11} & -Z_{21}^\top \\ Z_{21} & Z_{22} \end{psmallmatrix}$. The block form of $\nu$ is:
\begin{equation} \label{eq:target_block}
    \nu_{\mathcal{B}} = [Z, P]_{\mathcal{B}} + \begin{pmatrix} \dot{\Lambda} & 0 \\ 0 & 0 \end{pmatrix} 
    = \begin{pmatrix} \dot{\Lambda} + [Z_{11}, \Lambda] & \Lambda Z_{21}^\top \\ Z_{21} \Lambda & 0 \end{pmatrix}.
\end{equation}
Note that the $(2,2)$ block vanishes because $P$ is zero on the subspace orthogonal to $U$.

\paragraph{Step 2: Computation of the Differential $\eta$.}
We compute $\eta = M + M^\top$ where $M = (I - P) \Delta K X^\top$.
First, using the SVD $X = U \Sigma V^\top$, we simplify the term $K X^\top$:
\[
    K X^\top = (X^\top X + \varepsilon I)^{-1} X^\top 
    = V (\Sigma^2 + \varepsilon I)^{-1} \Sigma U^\top 
    = V D_\sigma U^\top,
\]
where $D_\sigma$ is diagonal with entries $(D_\sigma)_{ii} = \frac{\sigma_i}{\sigma_i^2 + \varepsilon}$.
Next, substituting $\Delta = U \Omega_A V^\top + U_\perp \Omega_B V^\top$ and applying the projection $(I-P)U = U(I-\Lambda)$:
\begin{align*}
    M &= \left[ U(I-\Lambda)U^\top + U_\perp U_\perp^\top \right] \left( U \Omega_A V^\top + U_\perp \Omega_B V^\top \right) V D_\sigma U^\top \\
      &= \left[ U(I-\Lambda)\Omega_A + U_\perp \Omega_B \right] D_\sigma U^\top.
\end{align*}
In the basis $\mathcal{B}$, $M$ and the resulting symmetric $\eta = M + M^\top$ are:
\[
    M_{\mathcal{B}} = \begin{pmatrix} (I-\Lambda)\Omega_A D_\sigma & 0 \\ \Omega_B D_\sigma & 0 \end{pmatrix}, \quad 
    \eta_{\mathcal{B}} = \begin{pmatrix} (I-\Lambda)\Omega_A D_\sigma + D_\sigma \Omega_A^\top (I-\Lambda) & D_\sigma \Omega_B^\top \\ \Omega_B D_\sigma & 0 \end{pmatrix}.
\]

\paragraph{Step 3: Matching and Scalar Conversion.}
Comparing the blocks of $\nu_{\mathcal{B}}$ (Eq. \ref{eq:target_block}) and $\eta_{\mathcal{B}}$ yields the solutions.

\textbf{1. Subspace Rotation ($Z_{21}$):} Matching the $(2,1)$ blocks gives $Z_{21} \Lambda = \Omega_B D_\sigma$, or $Z_{21} = \Omega_B D_\sigma \Lambda^{-1}$. 
Using the relation $\lambda_i = \frac{\sigma_i^2}{\sigma_i^2 + \varepsilon}$, we derive the scalar conversion:
\[
    (D_\sigma \Lambda^{-1})_{ii} = \frac{\sigma_i}{\lambda_i(\sigma_i^2 + \varepsilon)} = \frac{1}{\sigma_i} = \frac{1}{\sqrt{\varepsilon}} \sqrt{\frac{1-\lambda_i}{\lambda_i}}.
\]
Thus, $Z_{21} = \frac{1}{\sqrt{\varepsilon}} \Omega_B \operatorname{diag}\left( \sqrt{\frac{1-\lambda_i}{\lambda_i}} \right)$.

\textbf{2. Eigenvalue Variation ($\dot{\Lambda}$):} Matching the diagonal of the $(1,1)$ blocks (where $[Z_{11}, \Lambda]$ vanishes):
\[
    \dot{\lambda}_i = 2 (1-\lambda_i) (\Omega_A)_{ii} (D_\sigma)_{ii}.
\]
Using $(D_\sigma)_{ii} = \lambda_i / \sigma_i$ and the previous expression for $1/\sigma_i$:
\[
    \dot{\lambda}_i = 2 (1-\lambda_i) \lambda_i (\Omega_A)_{ii} \left[ \frac{1}{\sqrt{\varepsilon}} \frac{\sqrt{1-\lambda_i}}{\sqrt{\lambda_i}} \right] 
    = \frac{2}{\sqrt{\varepsilon}} \sqrt{\lambda_i} (1-\lambda_i)^{3/2} (\Omega_A)_{ii}.
\]
\end{proof}

\subsection{Optimization Landscape}

We study the quadratic approximation objective $f: \overline{\mathcal{M}} \to \R$ for a target vector $y \in \R^n$. Using the spectral form $P = U \Lambda U^\top$,  one gets:
$$
\begin{array}{rl}
    f(P) =& \frac{1}{2} \| (I_n - P) y \|^2 \\
    =& \frac{1}{2} \| (I_n - U \Lambda U^\top) y \|^2\\
    =& \frac{1}{2} \| U(I_n - \Lambda) U^\top y \|^2 + \frac{1}{2}  \| (I_n - UU^\top) y \|^2\\
    =& \frac{1}{2}\sum_{i=1}^p (1 - \lambda_i)^2 (u_i^\top y)^2 + \frac{1}{2}  \| (I_n - UU^\top) y )\|^2.
\end{array}
$$
where $u_i$ are the columns of $U$.

\subsubsection{Critical Points and Stability}

We analyze the critical points of $f$ to understand the convergence properties of optimization algorithms on $\mathcal{M}$.

\begin{proposition}[Gradient and Critical Points]
The gradient of $f$ vanishes in the interior $\mathcal{M}$ (where $\lambda_i < 1$) if and only if the subspace $U$ is orthogonal to the data vector $y$.
\begin{equation}
    \nabla f(P) = 0 \iff P y = 0.
\end{equation}
\end{proposition}

\begin{proof}
One has:
\[
\frac{\partial f}{\partial \lambda_i} = - (1-\lambda_i) (u_i^\top y)^2.
\]
For a critical point in the interior, we require $\frac{\partial f}{\partial \lambda_i} = 0$. Since $\lambda_i < 1$, this forces $(u_i^\top y)^2 = 0$ for all $i$. Therefore, $U^\top y = 0$, which implies $P y = U \Lambda U^\top y = 0$.
\end{proof}

\begin{theorem}[Boundary Classification]\label{th:boundary_classif}
    Let $P = U \Lambda U^\top \in \overline{\mathcal{M}}$ be a point on the boundary. The point $P$ is a critical point of the objective $f(P) = \frac{1}{2}\|(I-P)y\|^2$ if and only if the subspace $U$ is a critical point of the projection error on the Grassmannian.
    Specifically:
    \begin{enumerate}
        \item If $\Image(U)$ contains $y$, then $P$ is a \textbf{Global Minimum}.
        \item If $\Image(U)$ does not contain $y$ (but is a critical point), then $P$ is an \textbf{Unstable Saddle Point}.
    \end{enumerate}
\end{theorem}

\begin{proof}
    We analyze the local landscape using the decoupled objective form:
    \begin{equation}\label{eq:decoupled}
        f(U, \Lambda) = \frac{1}{2}\sum_{i=1}^p (1 - \lambda_i)^2 (u_i^\top y)^2 + \frac{1}{2}  \| (I_n - UU^\top) y )\|^2.
    \end{equation}

\paragraph{1. First-Order Optimality Conditions (Criticality)}
The tangent space variations allow us to differentiate with respect to $\lambda_i$ and the subspace $U$.

First, considering the derivative with respect to an eigenvalue $\lambda_i$:
\begin{equation}
    \frac{\partial f}{\partial \lambda_i} = -(1 - \lambda_i)(u_i^\top y)^2.
\end{equation}
For $P$ to be a critical point, this derivative must vanish. This yields the necessary condition:
$$
\forall i \in \{1, \dots, p\}: \quad \lambda_i = 1 \quad \text{or} \quad u_i^\top y = 0.
$$

\paragraph{2. Case: $\Image(U)$ contains $y$ (Global Minimum)}
If $y \in \Image(U)$, then the projection residual vanishes: $\| (I - UU^\top) y \|^2 = 0$. Furthermore, as $(U,\Lambda)$ is a critical point $\sum_{i=1}^p (1 - \lambda_i)^2 (u_i^\top y)^2 = 0$ so that $f(U,\Lambda) = \frac{1}{2}  \| (I_n - UU^\top y )\|^2 = 0$ which corresponds to a global optimum. 

\paragraph{3. Case: $\Image(U)$ does not contain $y$ (Unstable Saddle Point)}
Suppose $y \notin \Image(U)$. Then the residual vector $r = (I - UU^\top)y$ is non-zero. Note that $r \perp \Image(U)$.

Suppose that for all $1 \leq j \leq p$, $\lambda_j = 1$. In such case, $f(U,\Lambda) = \frac{1}{2} \| (I_n - U U^\top) y \|^2$. The element $U \in \St$ is then a critical value of $U \rightarrow \frac{1}{2} \| (I_n - U U^\top) y \|^2$. By Theorem \ref{th:landscape}, as $y \notin \Image(U)$ then $y \perp \Image(U)$ and so $f(U,\Lambda) = \frac{1}{2} \| y\|^2$ is maximal. 

Now, suppose that there exists $1 \leq j \leq p$ such that $\lambda_j \neq 1$, the condition derived in step 1 implies $u_j^\top y = 0$. This means $u_j \perp y$.

We verify the instability by examining a specific perturbation in the tangent space: rotating the eigenvector $u_j$ towards the residual $r$ and increasing its eigenvalue.

Let us rearrange the objective function slightly. Using the Pythagorean identity $\|y\|^2 = \sum (u_i^\top y)^2 + \|(I-UU^\top)y\|^2$, we can rewrite \eqref{eq:decoupled} as:
\begin{align}
    f(U, \Lambda) &= \frac{1}{2} \sum_{i=1}^p (1-\lambda_i)^2 (u_i^\top y)^2 + \frac{1}{2} \left( \|y\|^2 - \sum_{i=1}^p (u_i^\top y)^2 \right) \\
    &= \frac{1}{2} \|y\|^2 + \frac{1}{2} \sum_{i=1}^p \left[ (1-\lambda_i)^2 - 1 \right] (u_i^\top y)^2 \\
    &= \frac{1}{2} \|y\|^2 + \frac{1}{2} \sum_{i=1}^p (\lambda_i^2 - 2\lambda_i) (u_i^\top y)^2.
\end{align}

Consider the perturbation of the pair $(u_j, \lambda_j)$ parameterized by $t > 0$:
\begin{itemize}
    \item Rotate $u_j$ towards the normalized residual $\hat{r} = r/\|r\|$: 
    $$u_j(t) = u_j \cos t + \hat{r} \sin t.$$
    Since $u_j \perp y$ (from criticality) and $\hat{r} \parallel r$, the overlap becomes:
    $$u_j(t)^\top y = 0 \cdot \cos t + (\hat{r}^\top y) \sin t = \|r\| \sin t.$$
    \item Perturb the eigenvalue: $\lambda_j(t) = \lambda_j + \delta(t)$, where $\delta(t) > 0$.
\end{itemize}

The change in the objective function corresponding to index $j$ is proportional to:
$$\Delta f \propto (\lambda_j(t)^2 - 2\lambda_j(t)) (u_j(t)^\top y)^2.$$

We analyze the coefficient $C(\lambda) = \lambda^2 - 2\lambda$. The derivative is $C'(\lambda) = 2\lambda - 2$, which is strictly negative for $\lambda \in [0, 1)$.
\begin{itemize}
    \item If $\lambda_j \in (0, 1)$, then $\lambda_j^2 - 2\lambda_j < 0$. Simply rotating $u_j$ (keeping $\lambda_j$ fixed) makes $(u_j(t)^\top y)^2 > 0$, multiplying a negative coefficient. Thus $f$ decreases.
    \item If $\lambda_j = 0$ (boundary case), the coefficient is 0 initially. We couple the rotation with an increase in eigenvalue, setting $\lambda_j(t) = t$ (moving inside the feasible set from the boundary). Then:
    $$\lambda_j(t)^2 - 2\lambda_j(t) = t^2 - 2t \approx -2t \quad (\text{for small } t).$$
    The squared alignment is $(u_j(t)^\top y)^2 \approx t^2 \|r\|^2$.
    The dominant term in the change of $f$ is roughly $-t^3 \|r\|^2 < 0$.
\end{itemize}

In all sub-cases where $y \notin \Image(U)$, there exists a valid direction in the tangent cone (rotating an unused or partially used eigenvector towards the error residual) that strictly decreases the objective. Therefore, $P$ is an unstable saddle point.
\end{proof}

\begin{remark}[Coupling with Neural Networks]
The favorable landscape properties described in Theorem \ref{th:boundary_classif} (global minima and unstable saddles) rely on the hypothesis that one can vary the eigenvalues $\lambda_i$ and the subspace $U$ independently on the manifold.

However, when $X$ is parametrized by a neural network with weights $\theta$ (i.e., $X = X(\theta)$), the variations of $\Lambda$ and $U$ are coupled through the Jacobian of the mapping $\theta \mapsto X$.
Let $\Delta \in T_X \R^{n \times p}$ be a variation in the output space induced by a weight update. From Corollary \ref{cor:exp_tangentspace_parameter}, the induced motion on the manifold is:
\begin{itemize}
    \item \textbf{Eigenvalue variations} $\dot{\Lambda}$ are driven by the diagonal part $\operatorname{diag}(\Omega_A)$, where $\Omega_A = U^\top \Delta V$.
    \item \textbf{Subspace rotations} $[Z, P]$ are driven by the off-diagonal part $\Omega_B = U_\perp^\top \Delta V$ (and the off-diagonals of $\Omega_A$).
\end{itemize}

To recover the landscape guarantees, the parametrization must be rich enough to decouple these controls \textit{i.e.} the mapping $d\theta \rightarrow (\Omega_A(d\theta) = U^\top (dX(\theta)[d\theta]) V , \Omega_B(d\theta) =  U_\perp^\top (dX(\theta)[d\theta]) V)$ must be surjective. In particular this is verified if $\theta \rightarrow X(\theta)$ is a submersion. 

%In particular, we require:
%\begin{enumerate}
%    \item \textbf{Eigenvalue Independence:} To ensure convergence of the singular values (satisfying $\lambda_i \to 1$ or $u_i^\top y \to 0$), the map $d\theta \mapsto \operatorname{diag}(U^\top (dX(\theta)[d\theta]) V)$ must be surjective onto the space of diagonal matrices.
%    
%    \item \textbf{Subspace Independence:} To escape saddle points where $y \notin \operatorname{Im}(U)$ (Theorem \ref{th:boundary_classif}), the network must be able to rotate the subspace towards the residual. This requires the map $d\theta \mapsto U_\perp^\top (dX(\theta)[d\theta]) V$ to be surjective onto $\R^{(n-p) \times p}$.
%\end{enumerate}
If the network architecture is too constrained (e.g., insufficient width), these surjectivity conditions may fail, creating "spurious" local minima induced by the parametrization rather than the objective itself.
\end{remark}

\section{Numerical results}\label{sec:num}
In this section, we present the practical relevance of the method on some simple toy problems tractable on a personal computer with simple architecture.  The  workstation used operates on Ubuntu (x86\_64) and is powered by an Intel Core i5-10300H processor (4 cores, 8 threads), featuring AVX2 instruction support (which is essential for accelerating TensorFlow operations on the CPU). The system is equipped with 32~GB of RAM, ensuring ample headroom for handling datasets and machine learning tasks efficiently. Tensorflow in CPU mode is used for computations.

\subsection{Regression cost}
In practice for a regression problem, one uses the protocol given in Algorithm 1.
\noindent\hrule height 1pt
\captionof{algorithm}{Deep VarPro Regression (Mini-batch)} \label{alg:regression}
\noindent\hrule

\begin{algorithmic}[1]
    \Require Training data $\mathcal{D} = \{(x_i, y_i)\}_{i=1}^{N}$, Batch size $B$, Learning rate $\eta$, Regularization $\varepsilon$.
    \State Initialize neural network parameters $\theta$ (Feature extractor $\Phi(x; \theta)$).
    \State Output dimension of $\Phi$ is $p$.
    
    \For{epoch $= 1$ to $E$}
        \State Shuffle $\mathcal{D}$ and split into batches.
        \For{each batch $(X_b, U_b) \in \mathcal{D}$}
            \State \textbf{1. Forward Pass (Basis Construction):}
            \State Compute feature matrix $\Phi_b = \Phi(X_b; \theta) \in \mathbb{R}^{B \times p}$.
    
            \State \textbf{2. Solve for Linear Coefficients ($\alpha$):}
            \State Construct matrix $M = \Phi_b^{\top} \Phi_b + \varepsilon I$.
            \State Compute RHS $v = \Phi_b^{\top} U_b$.
            \State Solve linear system $M \alpha^* = v$ for $\alpha^*$.
    
            \State \textbf{3. Loss Computation:}
            \State Compute projection $U_{pred} = \Phi_b \alpha^*$.
            \State Loss $\mathcal{L}(\theta) = \frac{1}{B} || U_b - U_{pred} ||^2$.
    
            \State \textbf{4. Backpropagation (Variable Projection):}
            \State Compute gradients $\nabla_\theta \mathcal{L}$.
            \State Update $\theta \leftarrow \theta - \eta \nabla_\theta \mathcal{L}$.
        \EndFor
        \State \textit{Optional: Compute validation loss with fixed $\alpha$ from training.}
    \EndFor
    
    \State \textbf{Final Step:}
    \State Compute global basis $\Phi_{all} = \Phi(X_{all}; \theta_{final})$.
    \State Solve for final global coefficients $\alpha_{final}$ on full dataset.
    \State \Return $\theta_{final}, \alpha_{final}$
\end{algorithmic}
\noindent\hrule

\begin{table}[H]
    \centering
    \caption{Simulation Parameters for regression}
    \label{tab:parameters_regression}
    \renewcommand{\arraystretch}{1.2} % Adds a little vertical breathing room
    \begin{tabular}{l c l}
        \toprule
        \textbf{Parameter} & \textbf{Value} & \textbf{Description} \\
        \midrule
        \multicolumn{3}{l}{\textit{Spatial Parameters}} \\
        $N$ & $10000$ & Number of collocation points \\
        $p$ & $100$ & Number of hidden neurons \\
        \midrule
        \multicolumn{3}{l}{\textit{Optimization Parameters}} \\
        Batch size & $512$ & Minibatch size \\
        $\eta$ (Learning rate) & $0.001$ & Learning rate for $w, b$ \\
        $ \varepsilon$ & $0.001$ & Tikhonov weight \\
        \bottomrule
    \end{tabular}
\end{table}

The parameters chosen to obtain results in Figure \ref{fig:regression_figs}, are given in Table \ref{tab:parameters_regression}. Note also that the network architecture is given by $p = 100$ features of the form $\sigma(w_j \cdot + b_j)$ where $\sigma$ is a regularized version of the hat activation function \textit{i.e}:

\begin{equation}\label{eq:hat_activation}
\sigma(x) \approx ReLU(1-|x|)
\end{equation}

\begin{figure}[H]
    \centering
    % --- Première figure ---
    \begin{subfigure}[b]{\textwidth}
        \centering
        \includegraphics[width=0.8\textwidth]{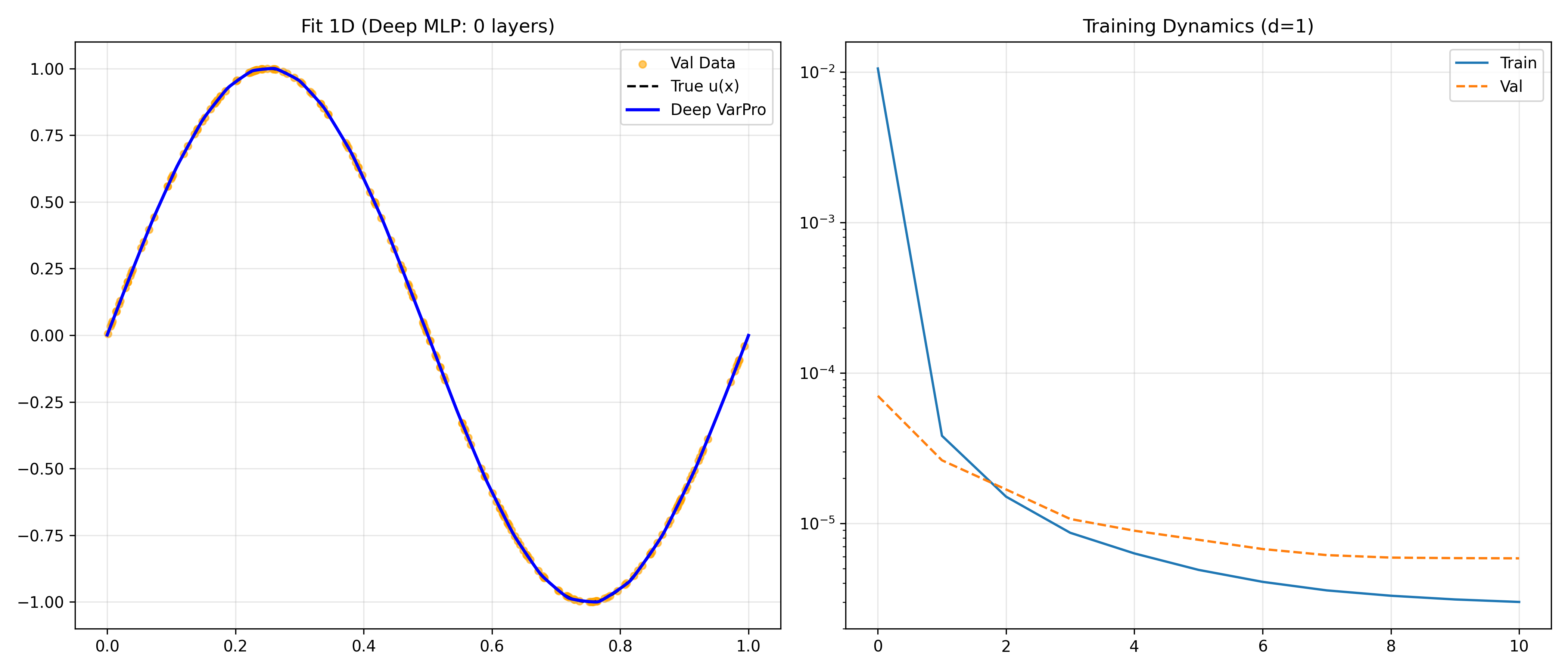}
        \caption{Case $d=1$, $f(x) = \sin(2 \pi x)$ }
        \label{fig:img1}
    \end{subfigure}
    
    \vspace{0.5cm} % Espacement vertical entre les figures

    % --- Deuxième figure ---
    \begin{subfigure}[b]{\textwidth}
        \centering
        \includegraphics[width=0.8\textwidth]{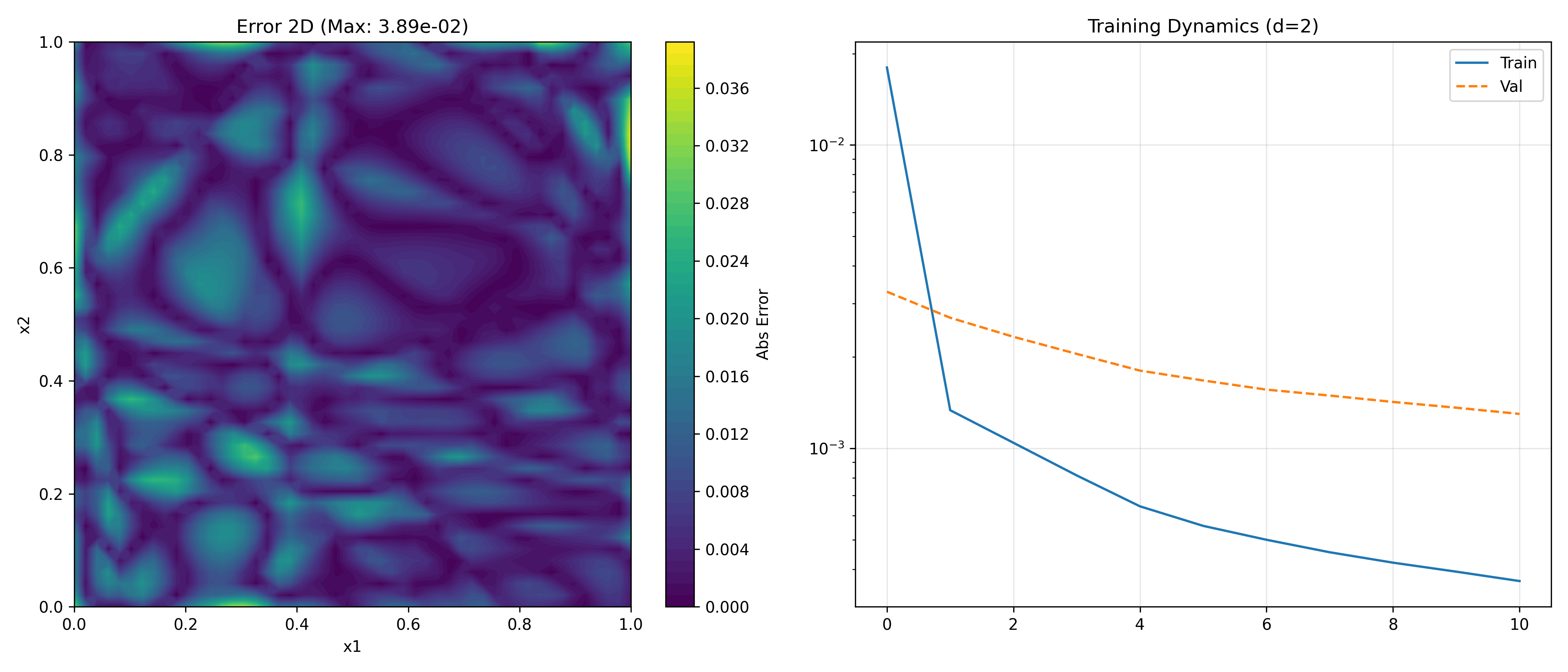}
        \caption{Case $d=2$, $f(x) = \sin(2 \pi x) + \sin(2 \pi y)$}
        \label{fig:img2}
    \end{subfigure}
    
    \caption{Results of VarPro on some regression problems}
    \label{fig:regression_figs}
\end{figure}
    % --- Troisième figure ---
\begin{figure}[H]
        \centering
        \includegraphics[width=0.5\textwidth]{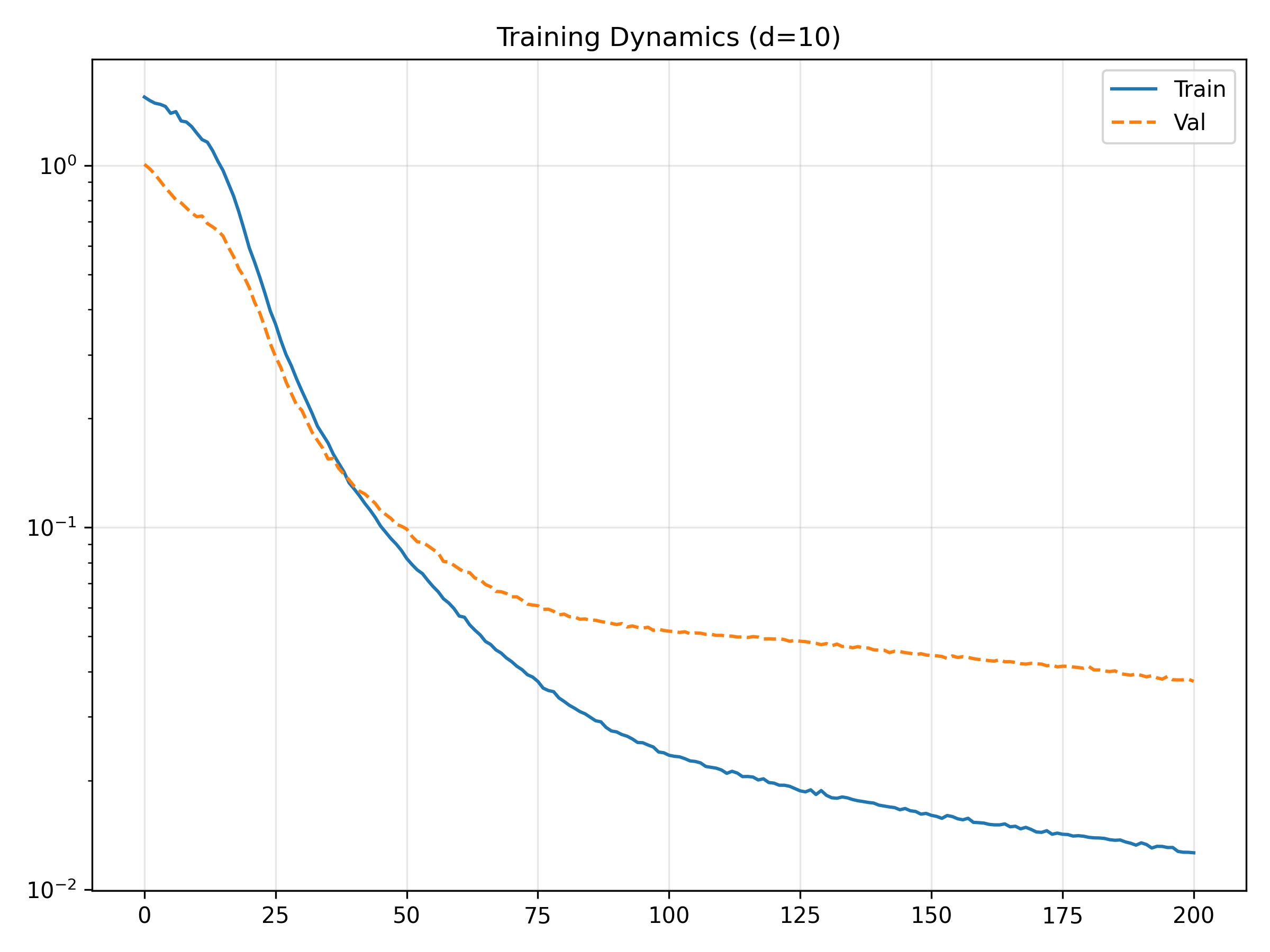}
        \caption{Case $d=10$, $f(x) = \sum_{i=1}^d \sin(2 \pi x_i)$}
        \label{fig:highdim_reg}
\end{figure}
We see in Figures \ref{fig:regression_figs}-\ref{fig:highdim_reg} that results are quite good even for high dimensional problems. Note that the regularization factor $\varepsilon$ must be large enough. Indeed if it is too small, one observes overfitting on Figure \ref{fig:overfitting}.

\begin{figure}[H]
        \centering
        \includegraphics[width=0.8\textwidth]{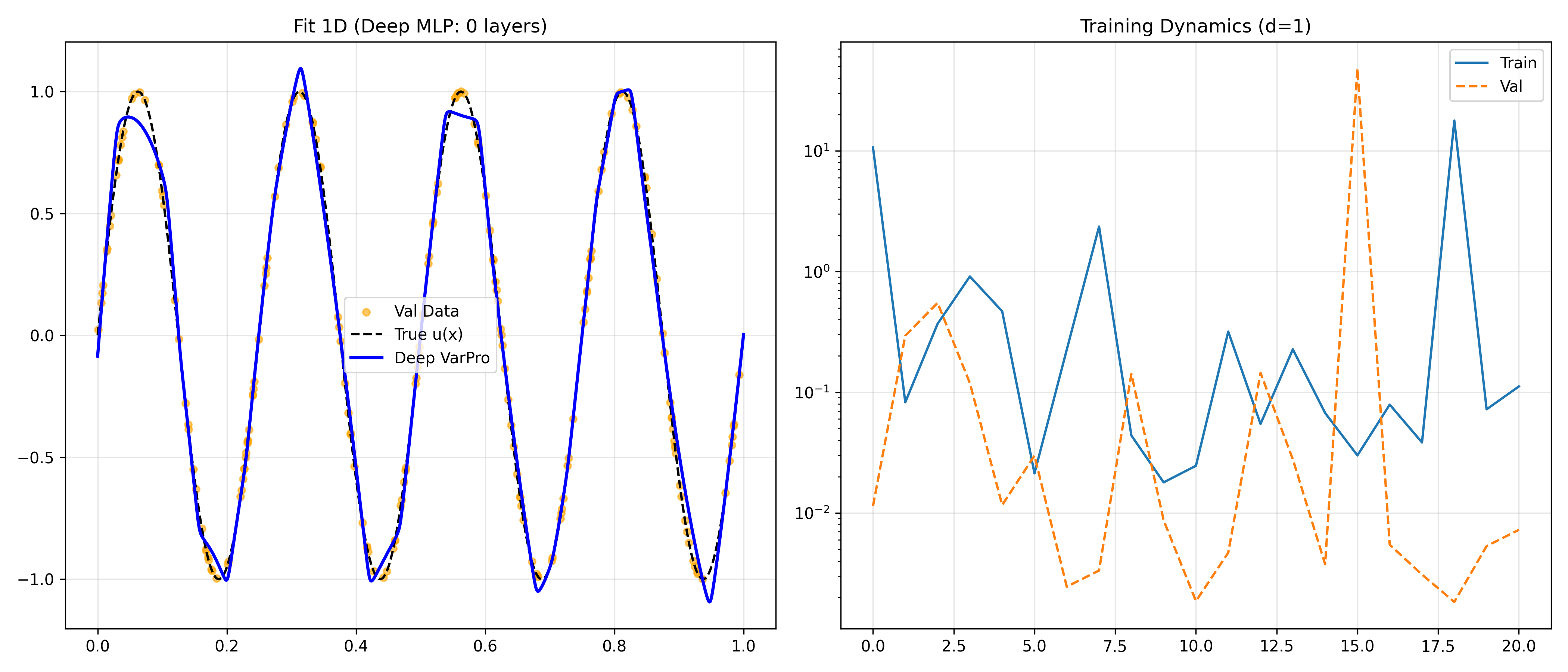}
        \caption{The case $d=1$ with $ \varepsilon= 10^{-16}, f(x) = \sin(8 \pi x)$}
        \label{fig:overfitting}
\end{figure}

\subsection{PINN cost}
As an example of PINN formulation, we try to solve a Poisson problem:
$$
- \Delta u = f \text{ on } \Omega = [0,1]^d 
$$
with Dirichlet boundary conditions. The source term $f$ is taken such that the solution is given by:

$$
u(x) := \sum_{i=1}^d \sin(\pi x_i).
$$
Here the network architecture is given by $p$ features of the form $\sigma(w_j \cdot + b_j)$ where $\sigma$ is the tanh activation function. The in domain and boundary batch size is $64$, the number of in domain sampled points is $512$ and so is the number of boundary sampled points. Finally, the boundary weight is $w_{bc} = 10$. The method shows good performance and is depicted in Algorithm 2 while results are given on Figure \ref{fig:poisson}.
\begin{figure}[H]
    \centering
    % --- Première figure ---
    \begin{subfigure}[b]{\textwidth}
        \centering
        \includegraphics[width=0.7\textwidth]{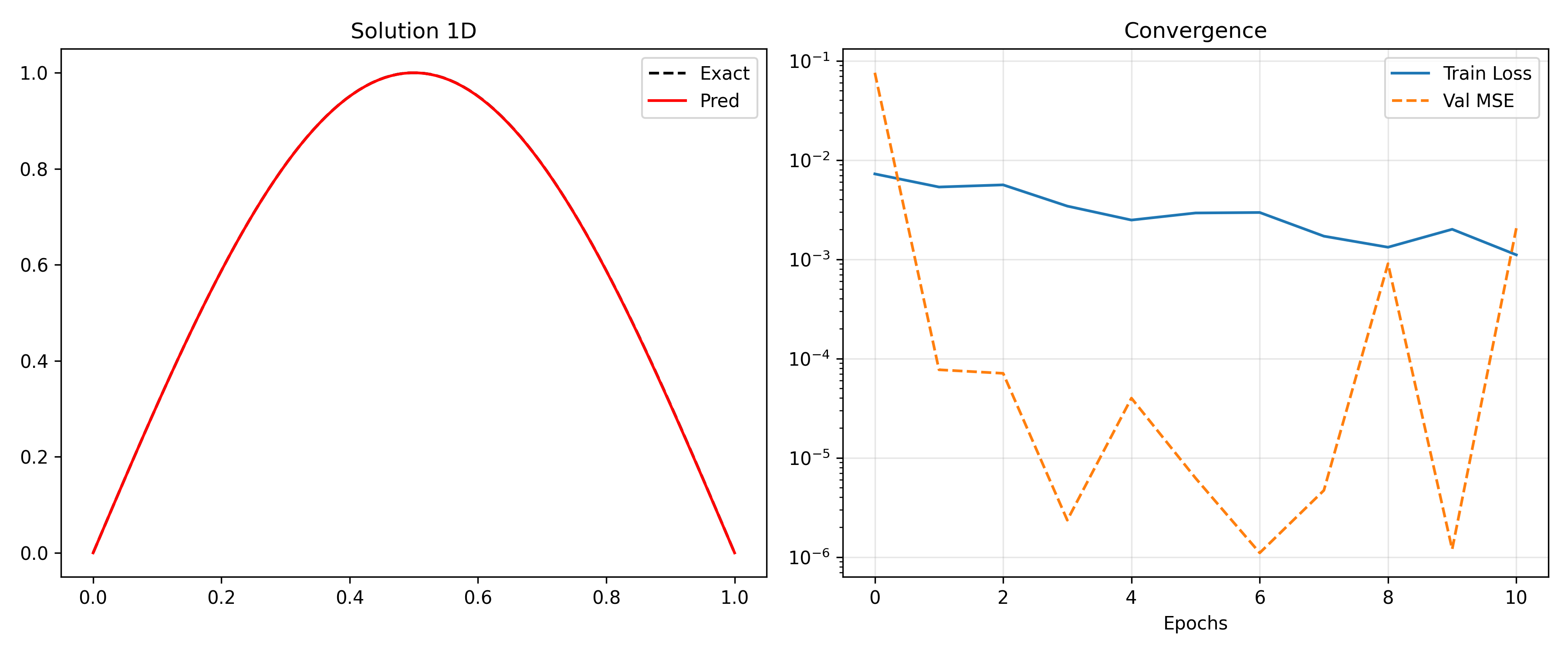}
        \caption{Case $d=1, p=20$ }
    \end{subfigure}
    
    \vspace{0.5cm} % Espacement vertical entre les figures

    % --- Deuxième figure ---
    \begin{subfigure}[b]{\textwidth}
        \centering
        \includegraphics[width=0.7\textwidth]{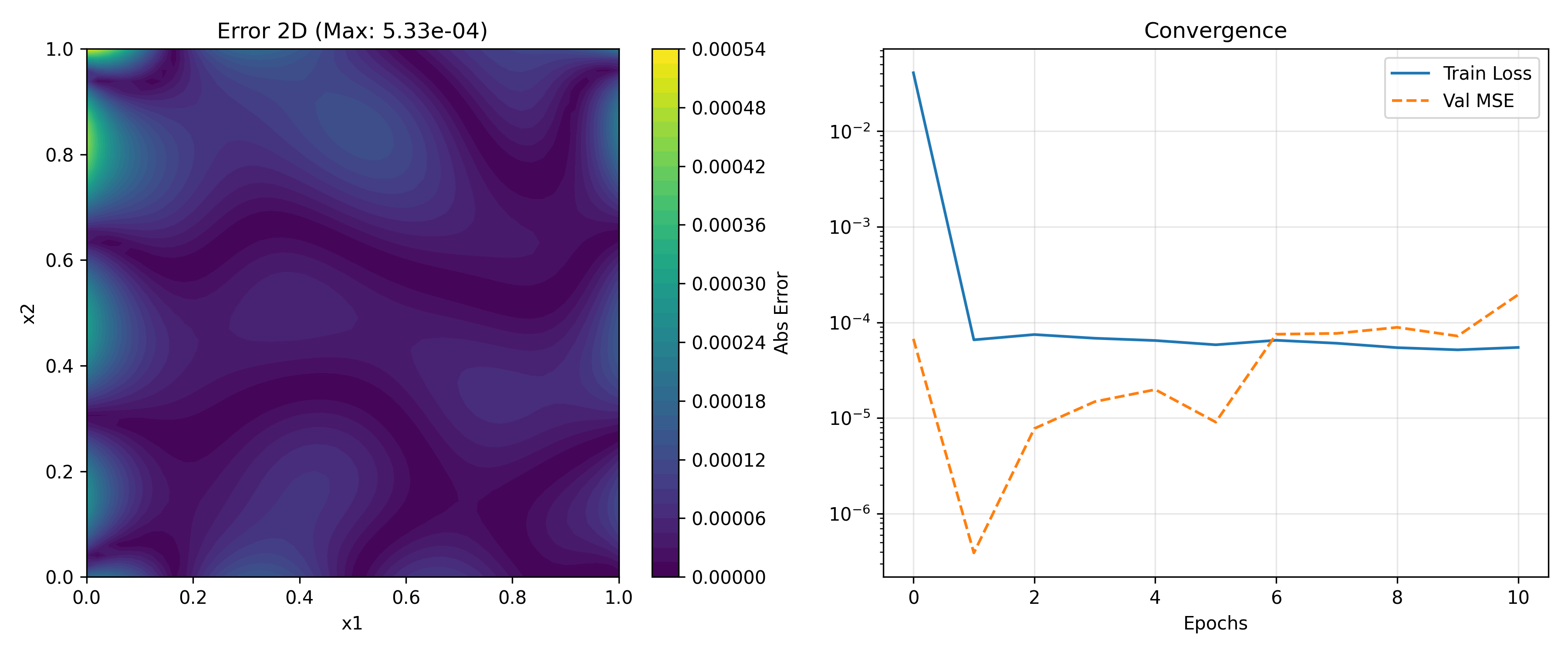}
        \caption{Case $d=2, p=50$}
    \end{subfigure}
    \caption{Results for some Poisson's problems}
    \label{fig:poisson}
\end{figure}

% --- Start of the Algorithm ---

% 1. Top visual rule
\noindent\hrule height 1pt 

% 2. Caption using captionof (allows the break)
\captionof{algorithm}{Deep VarPro Solver for Poisson Equation (Mini-batch)} \label{alg:laplace}

% 3. Middle visual rule
\noindent\hrule 

% 4. The content (automatically breaks across pages)
\begin{algorithmic}[1]
    \Require Domain $\Omega \subset \mathbb{R}^d$, Boundary $\partial \Omega$
    \Require PDE: $-\Delta u = f$ in $\Omega$, $u = g$ on $\partial \Omega$
    \Require Hyperparameters: Basis size $p$, Learning rate $\eta$, Weight $w_{bc}$, Batch size $B$
    \State Initialize Neural Network basis $\Phi(\cdot; \theta): \mathbb{R}^d \to \mathbb{R}^p$
    \State Generate collocation data $\mathcal{D}_{int} = \{x_i, f(x_i)\}_{i=1}^{N_{col}}$
    \State Generate boundary data $\mathcal{D}_{bnd} = \{x_j, g(x_j)\}_{j=1}^{N_{bnd}}$

    \While{epoch $<$ E}
        \State Shuffle $\mathcal{D}_{int}$ and $\mathcal{D}_{bnd}$
        \For{each minibatch $(X_c, F_c)$ from $\mathcal{D}_{int}$ and $(X_b, U_b)$ from $\mathcal{D}_{bnd}$}
            
            \State \Comment{\textbf{Step 1: Basis Evaluation \& Derivatives}}
            \State Compute $\Phi(X_c)$ and $\Phi(X_b)$ via forward pass
            \State Compute $\Delta \Phi(X_c)$ via Automatic Differentiation (Jacobian trace)
            
            \State \Comment{\textbf{Step 2: Linear Assembly (VarPro)}}
            \State Construct the design matrix $A$ and target vector $\mathbf{y}$:
            \State $A \leftarrow \begin{bmatrix} -\Delta \Phi(X_c) \\ w_{bc} \cdot \Phi(X_b) \end{bmatrix}$
            \State $y \leftarrow \begin{bmatrix} F_c \\ w_{bc} \cdot U_b \end{bmatrix}$
            
            \State \Comment{\textbf{Step 3: Solve Coefficients $\boldsymbol{\alpha}$}}
            \State $\alpha^* \leftarrow (A^{\top} A + \epsilon I)^{-1} A^{\top} y$ \Comment{Least Squares}
            
            \State \Comment{\textbf{Step 4: Gradient Descent on $\theta$}}
            \State Compute Loss $\mathcal{L}(\theta) = \frac{1}{B} || A \alpha^* - y ||_2^2$
            \State Compute gradients $\nabla_\theta \mathcal{L}$
            \State Update $\theta \leftarrow \theta - \eta \nabla_\theta \mathcal{L}$ \Comment{Adam Optimizer}
            
        \EndFor
        
        \State \Comment{Optional: Validation step}
        \State Update $\alpha_{final}$ using full dataset or large batch
    \EndWhile

    \Ensure Approximate solution $u(x) \approx \Phi(x; \theta) \alpha_{final}$
\end{algorithmic}
\noindent\hrule

It is even possible to solve the heat equation solving first a regression problem to approach the initial data and then perform a VarPro resolution of the following implicit Euler problem:

$$
(I - dt \Delta )u_{t+dt} = u_t  
$$
(associated to a boundary condition) at each time step. In Figure \ref{fig:heat_error}, we show the relevance of the method for the list of parameters given in Table \ref{tab:parameters}. In this test case, we solve:

$$
\left\{
\begin{array}{rl}
\partial_t u =& \Delta u\\
u(t=0,x) =& \sin(\pi x)
\end{array}
\right.
$$
so that the exact solution writes $u(t,x) = \sin(\pi x) e^{- \pi^2 t}$. The $L^2$ error increases through time and stays inferior to $\approx 10^{-2}$ until $T_{final} = 0.15 s$.

\begin{table}[ht]
    \centering
    \caption{Simulation Parameters for the implicit Euler scheme}
    \label{tab:parameters}
    \renewcommand{\arraystretch}{1.2} % Adds a little vertical breathing room
    \begin{tabular}{l c l}
        \toprule
        \textbf{Parameter} & \textbf{Value} & \textbf{Description} \\
        \midrule
        \multicolumn{3}{l}{\textit{Spatial Parameters}} \\
        $d$ & $1$ & Spatial dimension \\
        $N_{colloc}$ & $2000$ & Number of collocation points \\
        $N_{bnd}$ & $500$ & Number of boundary points \\
        $p$ & $40$ & Number of hidden neurons \\
        \midrule
        \multicolumn{3}{l}{\textit{Temporal Parameters}} \\
        $T_{final}$ & $0.15$ & Final simulation time \\
        $\Delta t$ & $0.01$ & Time step size \\
        \midrule
        \multicolumn{3}{l}{\textit{Optimization Parameters}} \\
        Batch size & $200$ & Minibatch size \\
        Epochs (init) & $50$ & Pre-training epochs ($t=0$) \\
        Epochs (step) & $50$ & Epochs per time step \\
        $\eta$ (Learning rate) & $0.005$ & Learning rate for $w, b$ \\
        $w_{bc}$ & $50.0$ & Boundary condition weight \\
        \bottomrule
    \end{tabular}
\end{table}

\begin{figure}[H]
        \centering
        \includegraphics[width=0.6\textwidth]{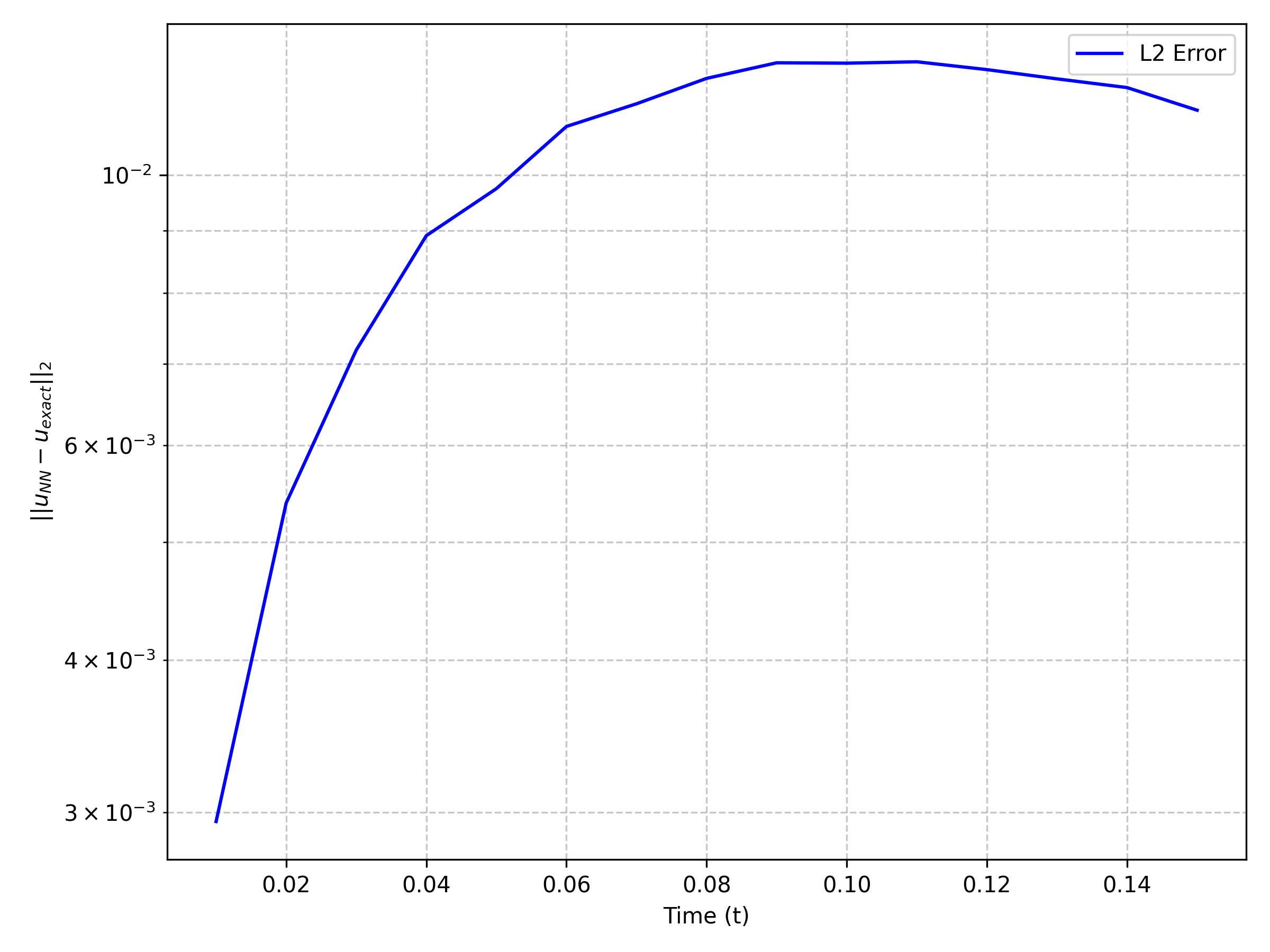}
        \caption{The $L^2$ error through time for the implicit Euler scheme}
        \label{fig:heat_error}
\end{figure}

We can compare its performance with a full VarPro PINNs (see Figure \ref{fig:heat_error_full_PINNs}) using Algorithm 3 for which parameters are given in Table \ref{tab:parameters_heat_pinns}. Results are comparable but the complexity of the network for Algorithm 3 is higher with 3 hidden layers of 30 neurons each.

\begin{figure}[H]
        \centering
        \includegraphics[width=0.6\textwidth]{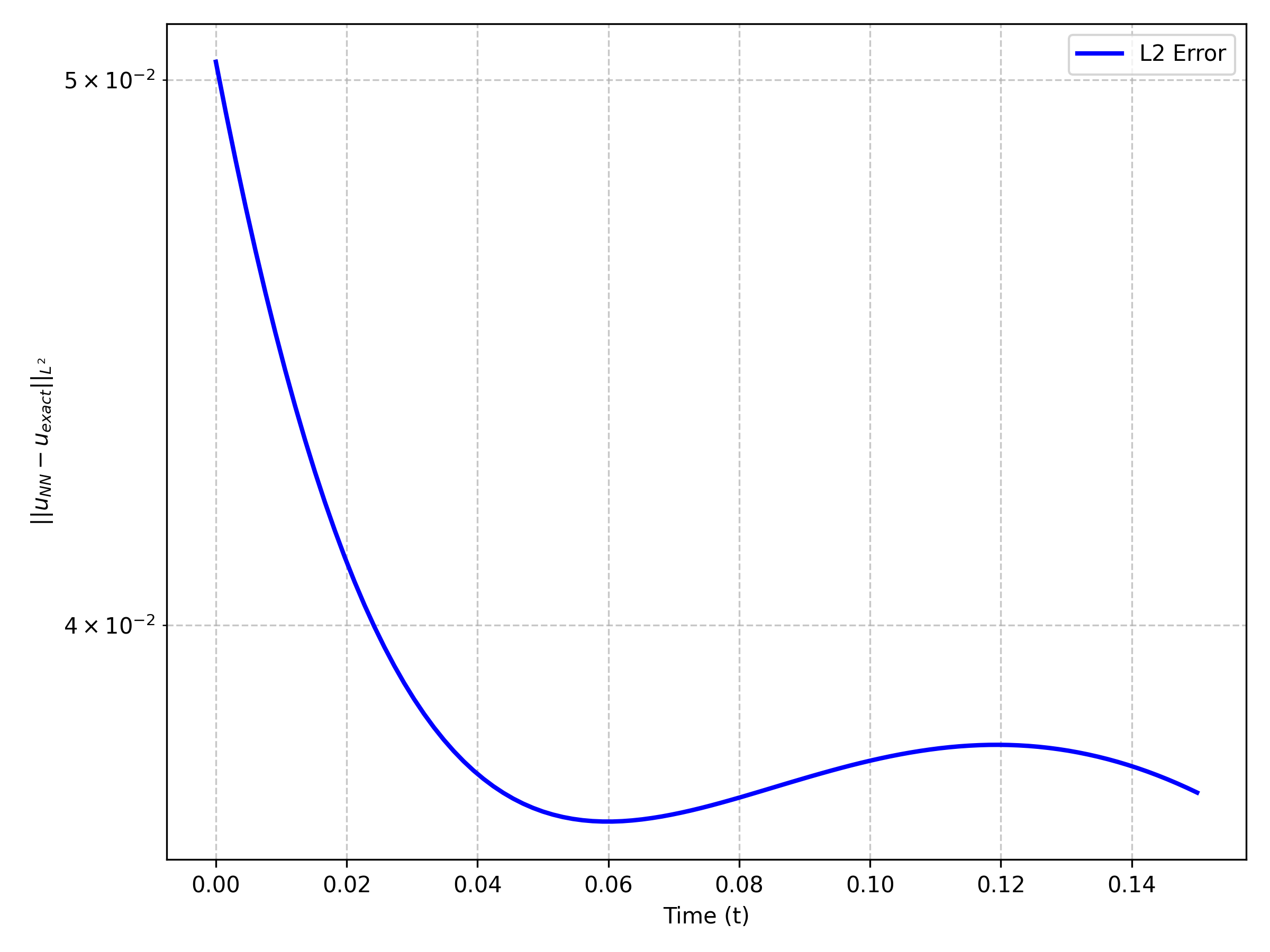}
        \caption{The $L^2$ error through time for the full PINNs formulation}
        \label{fig:heat_error_full_PINNs}
\end{figure}

\noindent\hrule height 1pt
\captionof{algorithm}{Deep VarPro Solver for Heat Equation (Mini-batch)} \label{alg:heat}
\noindent\hrule

\begin{algorithmic}[1]
    \Require Domain $\Omega \subset \mathbb{R}^d$, Time $T > 0$, Spacetime $\mathcal{D} = \Omega \times [0, T]$
    \Require PDE: $\partial_t u - \Delta u = f$ in $\mathcal{D}$
    \Require IC: $u(0, x) = u_0(x)$ for $x \in \Omega$
    \Require BC: $u(t, x) = g(t, x)$ on $\partial \Omega \times [0, T]$
    \Require Hyperparameters: Basis size $p$, Learning rate $\eta$, Weights $w_{ic}, w_{bc}$, Batch size $B$

    \State Initialize Neural Network basis $\Phi(\cdot; \theta): \mathbb{R}^{d+1} \to \mathbb{R}^p$
    \State Generate collocation data $\mathcal{D}_{pde} = \{(t_i, x_i), f_i\}_{i=1}^{N_{col}}$
    \State Generate initial data $\mathcal{D}_{ic} = \{(0, x_j), u_0(x_j)\}_{j=1}^{N_{ic}}$
    \State Generate boundary data $\mathcal{D}_{bc} = \{(t_k, x_k), g_k\}_{k=1}^{N_{bc}}$

    \While{epoch $<$ E}
        \State Shuffle datasets
        \For{minibatches $(X_p, F_p) \in \mathcal{D}_{pde}$, $(X_i, U_i) \in \mathcal{D}_{ic}$, $(X_b, U_b) \in \mathcal{D}_{bc}$}
            
            \State \Comment{\textbf{Step 1: Basis Evaluation \& Derivatives}}
            \State Compute $\Phi(X_p), \Phi(X_i), \Phi(X_b)$ via forward pass
            \State Compute $\partial_t \Phi(X_p)$ and $\Delta \Phi(X_p)$ via AutoDiff
            \State Compute PDE Operator: $\mathcal{L}[\Phi] = \partial_t \Phi(X_p) - \Delta \Phi(X_p)$
            
            \State \Comment{\textbf{Step 2: Linear Assembly (VarPro)}}
            \State Construct design matrix $A$ and target vector $\mathbf{y}$:
            \State $A \leftarrow \begin{bmatrix} \mathcal{L}[\Phi] \\ w_{ic} \cdot \Phi(X_i) \\ w_{bc} \cdot \Phi(X_b) \end{bmatrix}, \quad y \leftarrow \begin{bmatrix} F_p \\ w_{ic} \cdot U_i \\ w_{bc} \cdot U_b \end{bmatrix}$
            
            \State \Comment{\textbf{Step 3: Solve Coefficients $\boldsymbol{\alpha}$}}
            \State $\alpha^* \leftarrow (A^{\top} A + \epsilon I)^{-1} A^{\top} y$ \Comment{Least Squares}
            
            \State \Comment{\textbf{Step 4: Gradient Descent on $\theta$}}
            \State Compute Loss $\mathcal{L}(\theta) = \frac{1}{B} || A \alpha^* - y ||_2^2$
            \State Compute gradients $\nabla_\theta \mathcal{L}$
            \State Update $\theta \leftarrow \theta - \eta \nabla_\theta \mathcal{L}$ \Comment{Adam Optimizer}
            
        \EndFor
    \EndWhile

    \Ensure Approximate solution $u(t, x) \approx \Phi(t, x; \theta) \alpha_{final}$
\end{algorithmic}
\noindent\hrule

\begin{table}[ht]
    \centering
    \caption{Simulation Parameters for the Deep VarPro Heat Equation Solver}
    \label{tab:parameters_heat_pinns}
    \renewcommand{\arraystretch}{1.2} % Adds a little vertical breathing room
    \begin{tabular}{l c l}
        \toprule
        \textbf{Parameter} & \textbf{Value} & \textbf{Description} \\
        \midrule
        \multicolumn{3}{l}{\textit{Data Parameters}} \\
        $N_{col}$ & $10\,000$ & Interior collocation points \\
        $N_{init}$ & $2\,000$ & Initial condition points ($t=0$) \\
        $N_{bc}$ & $2\,000$ & Boundary condition points ($x \in \partial \Omega$) \\
        
        \midrule
        \multicolumn{3}{l}{\textit{Architecture Parameters}} \\
        $m$ & $30$ & Neurons per hidden layer \\
        $L$ & $3$ & Number of hidden layers \\
        $p$ (Basis) & $10$ & Size of neural basis $\Phi(x)$ \\
        $\sigma(\cdot)$ & $\tanh$ & Activation function \\
        
        \midrule
        \multicolumn{3}{l}{\textit{Optimization Parameters}} \\
        Batch ($B_{pde}$) & $256$ & Mini-batch size for PDE \\
        Batch ($B_{ic}, B_{bc}$) & $64$ & Mini-batch size for IC/BCs \\
        Epochs & $60$ & Total training epochs \\
        $\eta$ & $0.001$ & Adam learning rate \\
        $w_{ic}, w_{bc}$ & $100.0$ & Penalty weights for IC and BC \\
        $\lambda$ (Ridge) & Adaptive & $0.01 \times (0.95)^{\text{epoch}}$ \\
        \bottomrule
    \end{tabular}
\end{table}

\section{Conclusion and Perspectives}

We have presented a rigorous framework for Variable Projection in neural networks, framing the optimization of the linear output layer as a trajectory on the Grassmannian manifold. Our analysis guarantees global optimality in the over-parametrized regime and local stability for small residuals in the under-parametrized case. Furthermore, numerical evidence confirms that decoupling the linear and non-linear dynamics yields robust convergence properties.

Looking forward, this geometric perspective offers a bridge to recent advances in feature learning, particularly regarding time-scale separation. While \cite{Barboni2025} analyzes the ``two-timescale'' regime where linear weights evolve faster than hidden features, our VarPro approach can be interpreted as the rigorous limit of infinite time-scale separation. A natural next step is to extend this finite-dimensional Grassmannian framework to the infinite-width limit. By moving from a finite basis to kernel optimization within a Reproducing Kernel Hilbert Space (RKHS), one could potentially derive a variational characterization of the optimal kernel evolution, unifying the geometry of VarPro with the dynamics of feature learning.

\bibliographystyle{plain} % ou alpha, ieee, abbrv…
\bibliography{biblio}     % sans l’extension .bib

\end{document}